\documentclass[leqno]{amsart}
\usepackage[utf8]{inputenc}     
\usepackage[T1]{fontenc}
\usepackage{amsmath,mathtools, amssymb, amsfonts, latexsym, mathrsfs, pifont}
\makeatletter
\usepackage[inline]{enumitem} 
\setlist[itemize]{topsep=0pt,after=\vspace{1.5\baselineskip}}
\usepackage{color}
\usepackage[colorlinks=true]{hyperref}
\usepackage[margin=0.5in]{geometry}

\newtheorem{theorem}{Theorem}[section]

\newtheorem{lemma}[theorem]{Lemma}

\newtheorem{remark}{Remark}

\def\R{\mathbb R} \def\N{\mathbb N} 
\def\TM{T_{\textup{max}}} 

\def
\@cite
#1#2{[#1\if@tempswa, #2\fi]}
\makeatother

\usepackage{xparse}
\ExplSyntaxOn 
\int_new:N \g_tohi_const_int
\int_new:N \g_tohi_const_sub_int
\tl_new:N  \g_tohi_const_char_tl
\cs_new_protected:Nn \tohi_print_constant:nn 
{
  #1 \textsubscript {#2} 
} 

\NewDocumentCommand\resetconstants{m}
{
 \int_gincr:N \g_tohi_const_int
 \int_gzero:N \g_tohi_const_sub_int
 \tl_gset:Nn  \g_tohi_const_char_tl {#1}
}

\NewDocumentCommand\const{m}
{
  \tl_if_exist:cTF
   {
    c_tohi_const_\int_use:N\g_tohi_const_int _#1_tl
   }
   {
    \tl_use:c {c_tohi_const_\int_use:N\g_tohi_const_int _#1_tl }
   }
   {
    \int_gincr:N \g_tohi_const_sub_int
    \tl_const:cx {c_tohi_const_\int_use:N\g_tohi_const_int _#1_tl }
     { \exp_not:N\tohi_print_constant:nn {\g_tohi_const_char_tl }{\int_use:N \g_tohi_const_sub_int}}
    \tl_use:c {c_tohi_const_\int_use:N\g_tohi_const_int _#1_tl }
   }
}

\ExplSyntaxOff

\title[Positive total influx in two-species chemotaxis model]{
Global boundedness of a two-species attraction-attraction chemotaxis model with bilinear boundary influx}
\author[Gurusamy Arumugam, Silvia Frassu, Kwancheol Shin and Giuseppe Viglialoro]{$^{\dagger}$Gurusamy Arumugam, $^{\natural}$Silvia Frassu, $^{\flat}$Kwancheol Shin and $^{\natural, \star}$Giuseppe Viglialoro}
\makeatletter\@namedef{subjclassname@2020}{\textup{2020} Mathematics Subject Classification}\makeatother
\keywords{Chemotaxis, Two-species interaction, Robin-type boundary conditions, Gradient nonlinearities, Boundedness. \\
\textit{$^\star$Corresponding author}: giuseppe.viglialoro@unica.it}

\begin{document}

\maketitle
{
\centerline{$^{\dagger}$Department of Mathematics, School of Advanced Sciences}
\centerline{Vellore Institute of Technology}
\centerline{Vellore, Tamil Nadu, India}
\medskip
\centerline{$^{\natural}$Dipartimento di Matematica e Informatica}
\centerline{Universit\`{a} degli Studi di Cagliari}
\centerline{Via Ospedale 72, 09124. Cagliari, Italy}
\medskip
\centerline{$^{\flat}$Institute of Mathematical Science}
\centerline{Ewha Womans University}
\centerline{Seoul, South Korea}
\medskip
}
\bigskip

\resetconstants{d}

\begin{abstract}
Since its introduction, the Keller--Segel model has become a cornerstone in the mathematical theory of chemotaxis and it has generated extensive analytical activity. Most studies consider homogeneous Neumann boundary conditions, which ensure mass conservation and simplify the qualitative analysis of solutions.

To the best of the authors' knowledge, at present chemotaxis models incorporating boundary conditions that generate inward fluxes have only been studied in these recent papers \cite{Le2024} and \cite{BAG2026}, and we believe that this topic deserves and it may attract further mathematical attention. In this sense, in the present paper we investigate a two-species chemotaxis system with positive total flux. The model consists of two interacting populations, $u$ and $w$, coupled through elliptic/parabolic chemical signals $v$ and $z$, and subject to Robin-type boundary conditions allowing inward fluxes that depend on the product of the cellular and chemical densities. Unlike the classical conservative setting, the total mass is not preserved and it exhibits quadratic growth in time, exactly in line with \cite{BAG2026} dealing with a single-species taxis model. We show that, within the considered framework, standard logistic damping is not sufficient to compensate for the mass increase induced by the positive boundary flux. To restore control of the dynamics, stronger dissipative mechanisms involving gradient-dependent damping terms are required. Under suitable assumptions, we establish the global existence and boundedness of classical solutions in the presence of logistic-gradient damping. We finally emphasize that our contribution is not a straightforward extension of \cite{BAG2026}. Indeed, some cases appear to be difficult to handle due to the additional complexity of the two-species setting.
\end{abstract}
\tableofcontents
\section{The two-species chemotaxis phenomena: introduction and state of the art}
\subsection{The classical Keller--Segel models: some indications and considerations} 
 Chemotaxis refers to the directed movement of cells or microorganisms in response to spatial gradients of chemical substances. Cells migrate either toward higher concentrations of attractant chemicals (positive chemotaxis) or away from repellent substances (negative chemotaxis). This process plays a fundamental role in various biological phenomena, including immune response, embryonic development, tumor progression, and microbial aggregation.

In mathematical modeling, chemotaxis is commonly represented by a cross-diffusion term that describes the directed movement of a cell population along the gradient of a chemical concentration field. For example, in continuum models such as the classical Keller--Segel systems \cite{Keller-1971-TBC,Keller-1971-MC,K-S-1970}, the chemotactic driven-term is typically expressed as $
-\chi \nabla \cdot (u \nabla v)$,
where \(u\) denotes the cell density, \(v\) the chemical concentration, and \(\chi>0\) the chemotactic sensitivity coefficient. Introducing such a transport term into an otherwise purely diffusive model can significantly alter the qualitative behavior of solutions. While diffusion alone (i.e., $\Delta u$) generally promotes smoothing and dispersion, the aggregation induced by chemotaxis (i.e., $-\chi \nabla \cdot (u \nabla v)$) may lead to a wide range of dynamics. In particular, for the resulting system (corresponding to the classical Keller--Segel equation for the cells' distribution)
\[
u_t=\Delta u-\chi \nabla \cdot (u \nabla v),
\]
solutions $(u,v)=(u(x,t),v(x,t))$ in space ($x$) and time ($t$) may exist globally, remain bounded, and converge to steady states, or exhibit more complex long-time behaviors. On the other hand, when the aggregating effect dominates diffusion, solutions may exist only locally and blow up in finite time, a phenomenon in which the cell density becomes unbounded as the time approaches a finite critical instant. 

Since the introduction of the aforementioned Keller--Segel model in the 1970s, an extremely vast literature has developed (see Remark \ref{RemarkData} below), encompassing numerous variants, extensions, and related concepts. Given the breadth of this field, any selection of results is necessarily partial and could risk underrepresenting the diversity of existing studies. Therefore, we do not attempt a comprehensive survey; instead, we focus on works most relevant to the questions addressed in the present study. In particular, we refer to \cite{r13,BellomoEtAl} for overviews of the main results and of mentioned extensions such as nonlinear diffusion, volume-filling effects, logistic sources, multi-species interactions, fluid-coupled systems, and nonlocal or stochastic formulations.
\begin{remark}\label{RemarkData}
According to available bibliometric data from the literature, the Keller-–Segel model and its variants have generated a steadily increasing number of publications in pure and numerical mathematics since 1970. The output is limited in the 1970s–1980s, increases in the 1990s, reaches its maximum in the period 2000–2019, and remains high, though with a slightly lower annual rate, from 2020 onwards.
\end{remark}
A natural generalization of the classical Keller--Segel chemotaxis model involves systems of interacting species, whose dynamics are influenced by chemical signals as well as competitive interactions or external sources.
Two-species chemotaxis systems capture how competing populations interact and respond to shared chemical signals, a key question in mathematical biology with implications for spatial ecology and pattern formation. Since our work focuses on this class of models, we briefly outline the cases considered in the following sections:
\begin{itemize}
    \item [$\triangleright$] The two-species one-chemical chemotaxis system with competition (\S\ref{SecTwo-OneComp});
    \item [$\triangleright$] The two-species two-chemical chemotaxis system without competition (\S\ref{SecTwo-TwoWithout});
    \item [$\triangleright$] The two-species two-chemical chemotaxis system with competition or logistic terms (\S\ref{SecTwo-TwoCompetandLog}).
\end{itemize}
\subsection{The two-species one-chemical chemotaxis system with competition}\label{SecTwo-OneComp}
 Let us start by introducing the formulation of a two-species chemotaxis phenomena with single chemical: 
\begin{equation}
\begin{cases}
u_t = \Delta u - \nabla \cdot \left( u\chi_1(w)\nabla w \right) +u(a_1 - b_1 u - c_1 v), & \textrm{on } \Omega\times (0,T_{\rm{max}}), \\
v_t = \Delta v - \nabla \cdot \left( v\chi_2(w) \nabla w \right) + v(a_2 - b_2 v - c_2 u), & \textrm{on } \Omega\times (0,T_{\rm{max}}),\\
\tau w_t = \Delta w - \mu w + \nu u + \lambda v, & \textrm{on } \Omega\times (0,T_{\rm{max}}), \\
u_\nu =v_\nu = w_\nu = 0, & \textrm{on } \partial \Omega\times (0,T_{\rm{max}}),\\
u(x,0)=u_0(x), v(x,0)=v_0(x),\tau w(x,0)=\tau w_0(x) & \textrm{for } x \in \bar{\Omega}.
\end{cases}\label{t}
\end{equation}
(Although we imagine that readers familiar with chemotaxis systems may already be well acquainted with the structure and meaning of the model, we nevertheless briefly describe it for completeness and clarity.)  System \eqref{t} describes the interaction between two competing species, \(u=u(x,t)\) and \(v=v(x,t)\), coupled through a chemical signal \(w=w(x,t)\) produced by both populations. Here, \(\Omega \subset \mathbb{R}^n\) denotes a bounded domain with smooth boundary, while \(t \in (0,T_{\rm{max}})\), where \(T_{\rm{max}}\) represents the maximal time of existence of solutions; $\tau$ may assume values $0$ or $1$, and all the coefficients are positive. The terms \(\Delta u\) and \(\Delta v\) account for random diffusion of the species, whereas the cross-diffusive fluxes
\(-\nabla \cdot \left( u\chi_1(w)\nabla w \right)\) and \(-\nabla \cdot \left( v\chi_2(w)\nabla w \right)\)
model chemotactic movement directed by the gradient of the chemical substance \(w\). The logistic/competitions growth components
\(u(a_1 - b_1 u - c_1 v)\) and \(v(a_2 - b_2 v - c_2 u)\) describe intra- and interspecific competition. The chemical signal evolves according to
\(\tau w_t = \Delta w - \mu w + \nu u + \lambda v\),
where \(\mu>0\) represents natural decay and \(\nu, \lambda >0\) measure the contribution of each species to chemical production. The homogeneous Neumann boundary conditions (i.e., $u_\nu =v_\nu = w_\nu = 0$ on $\partial \Omega \times (0,T_{\rm{max}})$, where $(\cdot)_\nu$ indicates the outward normal derivative on $\partial \Omega$) ensure no flux across \(\partial \Omega\), and the system is completed with nonnegative initial data \(u_0, v_0, \tau w_0\). 

We briefly review some representative results concerning system \eqref{t}, with particular emphasis on the global existence, boundedness, stabilization, and asymptotic behavior of solutions. The qualitative dynamics are known to depend sensitively on the interplay between the system parameters, the spatial dimension, and the initial data. Among the factors governing the evolution, a prominent role is played by the chemotactic sensitivities, the strength of the logistic damping and competitive interactions, and the total initial mass. The existing literature may be broadly divided according to whether $\tau=0$ or $\tau=1$.
\begin{itemize}
    \item[$\triangleright$] 
    In the parabolic--elliptic setting ($\tau=0$), the case of constant sensitivities, namely
    $\chi_i(w)=\chi_i>0$ ($i=1,2$), has been extensively investigated; see, for instance,
    \cite{r38,r54,r62,r55}. Related results for signal-dependent sensitivities, for instance of the form
    $\chi_i(w)=\frac{\chi_i}{w}$, can be found in \cite{r45,r61,r53}.
    \item[$\triangleright$] 
    In the fully parabolic framework ($\tau=1$), results corresponding to constant sensitivities are available in
    \cite{r39,r43,r42,r63}. For signal-dependent sensitivities, singular or not, we refer to
    \cite{r64,r52,r53,r45,r46,r47,r49}.
\end{itemize} 


\subsection{The two-species two-chemical chemotaxis system without competition}\label{SecTwo-TwoWithout} 
In the previous section, boundedness of solutions was ensured by logistic damping and interspecific competition for a system with a single signal. Here, we consider a system with two chemical signals and no logistic damping. Without these stabilizing effects, diffusion may fail to prevent aggregation: large initial data or strong chemotactic sensitivities can lead to blow-up, while solutions emanating from small data are expected to remain globally bounded. 

To establish a starting point, let us consider the following two-species two-chemicals system:
\begin{align}
\begin{cases}
      u_{t} =\nabla\cdot( D_1(u)\nabla u) - \chi \nabla\cdot(S_1(u)\nabla v),& \textrm{on } \Omega\times (0,T_{\rm{max}}), \\
       0= \Delta v - \alpha v+\beta w, & \textrm{on } \Omega\times (0,T_{\rm{max}}), \\
      w_{t} =\nabla\cdot(D_2(w)\nabla w)- \xi \nabla\cdot(S_2(z)\nabla z),& \textrm{on } \Omega\times (0,T_{\rm{max}}), \\ 
    0=\Delta z -\gamma z+\delta u, & \textrm{on } \Omega\times (0,T_{\rm{max}}), 
\end{cases}\label{rf}
\end{align}
in which we have omitted to specify the homogeneous Neumann boundary conditions and the initial data. In \cite{r5} the authors consider the problem in a bounded domain $\Omega \subset \mathbb{R}^2$, and for $D_i(s)=\alpha=\gamma=1$, $S_i(s)=s$ with $s \geq 0$, $i=1,2$, they obtain, inter alia, that for $
m_1 m_2 - 2\pi (m_1 \chi \beta + m_2 \xi \delta) > 0,$
where $
m_1 := \int_\Omega u_0(x) \, dx$, $m_2 := \int_\Omega w_0(x) \, dx$
there exist finite-time blow-up solutions to the system. On the other hand, when $\chi = \xi = \beta = \delta = 1$, the blow-up criterion reduces to 
$m_1 m_2 - 2\pi (m_1 + m_2) > 0$, and global boundedness of solutions is furthermore established for $\alpha = \gamma = 1$ under the condition $
\max\{m_1, m_2\} < 4\pi$. These results improve upon those in \cite{r6}, where global solvability holds only under the assumption that the total mass $m_1+m_2$ is below a certain threshold for $n=2$ or it is sufficiently small for $n\geq 3$, while finite-time blow-up can occur for large mass in two dimensions or for arbitrarily small mass when $n\geq 3$. On the other hand, leaving aside the case of linear diffusion, for $D_i(u) = (u+1)^{p_i-1}$ and $S_i(u) = u(u+1)^{q_i-1}$, the authors in \cite{r9} proved global boundedness of solutions to model \eqref{rf} under the conditions $q_1 < \frac{2}{n} + p_1 - 1$ and $q_2 < \frac{2}{n} + p_2 - 1$. 

For the fully parabolic version of system \eqref{rf} (in which the second and forth equations are parabolic), with $S_i(s) = s$ ($i=1,2$) and diffusion coefficients $D_1(s)$ and $D_2(s)$ behaving as $s^{m_1-1}$ and $s^{m_2-1}$, with $m_1, m_2 > 1$ respectively,  global boundedness of solutions was proved for suitable initial data, provided $m_1 > 2 - \frac{2}{n}$ and $m_2 > 2 - \frac{2}{n}$. 

Regarding results close to those presented, if we consider model \eqref{rf}, but where the equations for the chemical signals are replaced by
\begin{equation}\label{ChemicalConMedia}  
0 = \Delta v - \mu_2 + \beta w, \quad \mu_2 = \int_\Omega w, 
\qquad \text{and} \qquad 
0 = \Delta z - \mu_1 + \delta u, \quad \mu_1 = \int_\Omega u,
\end{equation}
we can mention that finite-time blow-up for the case \(D_1(s) \simeq s^{m_1-1}\) and \(D_2(s) \simeq s^{m_2-1}\), with \(m_1, m_2 > 1\), and $S_i(s) = s$ ($i=1,2$), has been detected in \cite{r8} under the condition
\[
m_1 + m_2 > \max \left\{ \frac{m_1 m_2 + 2 m_1}{n}, \frac{m_1 m_2 + 2 m_2}{n} \right\}.
\]
In this same context, i.e. model \eqref{rf} with equations for the chemicals as in  \eqref{ChemicalConMedia}, in \cite{r10} the authors investigate the situation where with $S_1(s)\simeq s^p$ and $S_2(s)\simeq s^q$, $p,q>0$ and $D_i(s)=1$, for  $i=1,2$. They established sharp threshold conditions separating global boundedness from finite-time blow-up, thereby characterizing the interplay between chemotactic attraction and nonlinear diffusion. (Further related investigations on comparable models have also been conducted in \cite{r7,r12,r66}.)
\subsection{The two-species two-chemical chemotaxis system with competition or with logistics}\label{SecTwo-TwoCompetandLog}
In preparation for the main model studied in this work, we first review results on chemotaxis systems involving two species and two chemical signals, with either competitive Lotka–Volterra interactions or classical logistic growth. While one-species chemotaxis systems have been extensively studied in terms of blow-up phenomena and global boundedness, the role of logistic or sub-logistic sources in preventing blow-up in two-species systems remains less explored. The added complexity of multiple interacting species and nonlinear source terms makes rigorous analysis more challenging, so most studies focus on globality issues. To be more precise, if we take into consideration the following system
\begin{equation}\label{rf3First}
\begin{cases}
u_{t} = \Delta u - \chi_1\nabla\cdot(u\nabla v) + \mu_1 u(1 - u - a_1 w), & \textrm{on } \Omega \times (0,T_{\rm max}), \\[1ex]
\tau v_{t} = \Delta v - v + w, & \textrm{on } \Omega \times (0,T_{\rm max}), \\[0.5ex]
w_{t} = \Delta w - \chi_2 \nabla\cdot(w \nabla z) + \mu_2 w(1 - w - a_2 u), & \textrm{on } \Omega \times (0,T_{\rm max}), \\[0.5ex]
\tau z_{t} = \Delta z - z + u, & \textrm{on } \Omega \times (0,T_{\rm max}),
\end{cases}
\end{equation}
questions such as global existence and asymptotic behavior of solutions, global stability of positive steady states, and competitive exclusion principles (where one species may dominate and drive the other to extinction) have been investigated. Further works extend the analysis to chemotaxis-competition systems with two signals, identifying conditions for coexistence, competitive exclusion, and convergence toward steady states; see, for example, \cite{ZhangLiuYang2017,ZhengMuMi2018,Zhang2018Competitive,TuMuZhengLin2018}.

At this stage, in order to proceed toward the introduction of the model of primary interest in this work, let us consider first the system obtained by replacing the Lotka--Volterra type interaction terms $\mu_1 u(1 - u - a_1 w)$ and $\mu_2 w(1 - w - a_2 u)$ with classical logistic source terms
\begin{equation}\label{classicalLogistiSources}
a_1 u - b_1 u^2 \quad \text{and} \quad a_2 w - b_2 w^2, \qquad a_1,a_2,b_1,b_2>0,
\end{equation}
respectively. Then system \eqref{rf3First} reduces to a two-signal attraction-attraction (i.e., below, $\chi_1,\chi_2>0$) model with logistic growths (here $a_1,a_2,b_1,b_2>0$), representing our starting point and reading as (here, naturally, $\tau_i\in \{0,1\}, i=1,2$)
\begin{equation}\label{rf3}
\begin{cases}
u_{t} = \Delta u - \chi_1\nabla\cdot(u\nabla v) +a_1 u - b_1 u^2, & \textrm{on } \Omega \times (0,T_{\rm max}), \\[1ex]
\tau_1 v_{t} = \Delta v - v + w, & \textrm{on } \Omega \times (0,T_{\rm max}), \\[0.5ex]
w_{t} = \Delta w - \chi_2 \nabla\cdot(w \nabla z) +a_2 w - b_2 w^2, & \textrm{on } \Omega \times (0,T_{\rm max}), \\[0.5ex]
\tau_2 z_{t} = \Delta z - z + u, & \textrm{on } \Omega \times (0,T_{\rm max}).
\end{cases}
\end{equation}
Let us summarize some results from the literature concerning global solvability and boundedness of solutions. Specifically, for solutions to system \eqref{rf3} emanating from sufficiently regular and nonnegative initial data $
u_0(x), \tau_1 v_0(x), w_0(x), \tau_2 z_0(x)$,
it is known that $T_{\rm max} = \infty$ and that $u, v, w, z$ are uniformly bounded on $(0, \infty)$ under the following conditions:  
\begin{itemize}
    \item [$\triangleright$] For $\tau_1 = 1$ and $\tau_2 = 0$, $n\geq 1$, whenever for some existing $\theta_0=\theta_0(\Omega,n) > 0$, $b_i$ and $\chi_i$ are such that $
        \frac{b_1 b_2}{\chi_1 \chi_2} > \theta_0$ (
    see \cite{TianHeZheng2022});
    \item [$\triangleright$] For $\tau_1 = 1$ and $\tau_2 = 0$, $n\geq 1$, whenever for some existing $\theta_1=\theta_1(n,b_1,\chi_1,\chi_2) \geq  0$ and $b_2>\theta_1$ (see \cite{AyazogluEkincioglu2025}, and observe that this result partially improves \cite{TianHeZheng2022});
    \item [$\triangleright$] For the case $\tau_1 = \tau_2 = 0$, with $a_2 = b_2 = 0$ and $a_1 u - b_1 u^2$ replaced by $
        f(u) = a_1 u - \frac{b_1 u^2}{\ln^p(u+e)}$, with $p \in [0,1)$,
    in a two-dimensional setting (see \cite{Le2024}).
\end{itemize}
\section{Positive total influx in chemotaxis; results and presentation of the model}
\subsection{Positive fluxes and gradient-dependent logistic sources}
All the chemotaxis models discussed so far have been investigated under homogeneous Neumann  (\textit{no flux}) boundary conditions, which naturally reflect the conservation of mass within an isolated domain. In the present work, however, we are interested in the case of positive boundary fluxes. This naturally leads to the consideration of Robin boundary conditions, which provide a flexible framework for describing interactions between the population and its surrounding environment.

Robin boundary conditions are commonly employed to model complex phenomena, including species movement along boundaries. (Examples and applications for parabolic equations can be found in \cite{Souplet_Gradient,AndreGia,QS,SunYudong,PP,PS}.)

In the context of taxis-driven processes, where boundary responses to chemical gradients play an important role, Robin boundary conditions offer a more realistic description than purely Dirichlet or Neumann conditions. In this regard, only very recently Robin boundary conditions have been studied in the context of chemotaxis systems. In particular, \cite{BAG2026} analyzed the classical Keller--Segel model, in both parabolic--parabolic and parabolic--elliptic forms, under positive flux boundary conditions. As the author will recognize from the same work, the presence of such positive fluxes requires, in order to ensure boundedness, a sufficiently strong logistic term, including gradient-dependent effects, i.e. logistic of the type
\begin{equation}\label{logistiGradiente}
f(u,|\nabla u|) = a u - b u^2 - c |\nabla u|^2, \quad a,b,c>0.
\end{equation}
Since the problem considered here is closely related to that studied in \cite{BAG2026}, we would like to emphasize an aspect that was not addressed in that work, most likely because the authors were unaware of the recent contribution \cite{MinLeJDEPosFlux}. More precisely, in \cite{MinLeJDEPosFlux}  the author considers a boundary influx proportional to a subquadratic power of the cell density, whereas in \cite{BAG2026} the influx is proportional to the product of the cell and chemical densities, thus exhibiting essentially (bilinear/cross-species) quadratic growth. 

This distinction is analytically significant, as the strength of the logistic damping required to prevent blow-up is closely related to the growth of the boundary influx. Since the positive flux considered in \cite{MinLeJDEPosFlux} exhibits subquadratic growth, boundedness can be ensured by a  weaker logistic mechanism than \eqref{logistiGradiente}; in particular, no gradient-dependent damping is needed, corresponding to the case $c=0$. (We will review in more detail, in Remark \ref{RemakFlussiPositiviSingleSpecies}, the results from the works \cite{MinLeJDEPosFlux} and \cite{BAG2026}.)
\begin{remark}
For the sake of completeness, we point out that gradient-dependent logistic terms as those in \eqref{logistiGradiente} have not appeared in the chemotaxis literature solely to counterbalance positive boundary fluxes. They have also been introduced in the context of no-flux boundary conditions, in order to prevent blow-up when standard logistic damping is too weak to ensure boundedness (see \cite{IshidaLankeitVigliloro-Gradient,LiEtAL-SAM-2025}). This highlights the broader applicability of such terms in controlling population growth and ensuring global existence of solutions.
\end{remark}
Motivated by the aforementioned works, in this paper we consider \textit{an attraction-attraction two-species chemotaxis system with two productive chemicals and gradient-dependent logistic terms, subject to positive total flux boundary conditions.}  
To the best of the authors' knowledge, no study has yet addressed the global boundedness of solutions for this system. Apart from this important aspect, as will be shown, the model under consideration can be regarded as a generalization of the system formulated in \eqref{rf3}.
\subsection{Presentation of the model and main result}
In this paper, we study the following two-species chemotaxis phenomenon formulated by 
\begin{equation}\label{problem}
\begin{cases}
u_t =\Delta u - \chi \nabla\cdot(u\nabla v)+a_1u-b_1u^{2} -c_1|\nabla u|^{2} & \textrm{in } \Omega \times (0,\TM),\\  
\tau v_t= \Delta v - v+w &\textrm{in } \Omega \times (0,\TM),\\
w_{t} =\Delta w- \xi \nabla\cdot(w\nabla z)+a_2w-b_2w^{2} -c_2|\nabla w|^{2} &\textrm{in } \Omega \times (0,\TM),\\ 
 \tau z_t=\Delta z -z+u &\textrm{in } \Omega \times (0,\TM),\\
 u_{\nu} =(\alpha-1)h_{1}uv\chi,  \, v_{\nu} =-h_{1}v &\textrm{on } \partial \Omega \times (0,\TM),\\  
w_{\nu} =(\alpha-1)h_{2}wz\xi, \, z_{\nu} =-h_{2}z &\textrm{on } \partial \Omega \times (0,\TM),\\  
 u(x,0) =u_{0}(x), \tau v(x,0) =\tau v_{0}(x), w(x,0) = w_{0}(x), \tau z(x,0) =\tau z_{0}(x), & x \in \bar{\Omega},
\end{cases}
\end{equation}
where $\Omega \subset \R^n$, $n\geq 1$, is a bounded and smooth domain, $\chi, \xi, a_1,a_2, b_1, b_2, c_1, c_2, h_1, h_2 >0$, $\alpha \in [0,1]$ and $\tau \in \{0,1\}$. Moreover, $u_0, \tau v_0, w_0$ and $\tau z_0$ are sufficiently regular nonnegative initial data, while $\TM$ indicates the maximal existence time of solutions to the above problem and $\nu$ is the outward normal vector to the boundary of $\Omega$. (Let us observe that setting $c_1=c_2=h_1=h_2=0$ in model \eqref{problem}, the mechanism formulated in \eqref{rf3} is recovered; however, this limiting case is not of interest here, since our aim is precisely to highlight the interplay between positive values of the parameters $h_i$ and $c_i$.)

We will establish what follows: 
\begin{theorem}\label{theoremlocal}
For some $\delta\in(0,1)$ and $n\in \N$, let $\Omega \subset \R^n$ be a bounded domain of class $C^{2+\delta}$, $\tau=0$, $\chi, \xi, a_1,a_2, c_1, c_2, h_1, h_2 >0$ and $\alpha \in [0,1]$. Then there exists $C_{\partial \Omega}=C_{\partial \Omega}(\Omega, n)$ such that for every $u_0, w_0: \bar{\Omega}  \rightarrow \mathbb{R}^+$, with $u_0, w_0 \in C^{2+\delta}(\bar\Omega)$ complying with $u_{0\nu}=(\alpha-1)\chi h_1 u_0 v_0$ and $w_{0\nu}=(\alpha-1)\xi h_2 w_0z_0$ on $\partial \Omega$ the following conclusion holds true:
whenever either 
\begin{equation}\label{Estimate_b}
b_1> c_1 + \frac{(\xi \alpha)^2 h_2 C_{\partial \Omega}}{16 c_2}, \quad   
b_2 > c_2 + \frac{(\chi \alpha)^2 h_1 C_{\partial \Omega}}{16 c_1} \quad \textrm{(being $\alpha>0$}),  
\end{equation}
or
\begin{equation}\label{Estimate_Bag0}
\alpha \max\{\chi,\xi\} <\min\left\{4c_1, 4 c_2, \frac{2 b_1}{3},\frac{2 b_2}{3}\right\}, 
\end{equation}
problem \eqref{problem} admits a unique solution
\begin{equation*}
(u,v, w, z)\in C^{2+\delta,1+\frac{\delta}{2}}( \Bar{\Omega} \times [0, \infty))\times 
C^{2+\delta,\frac{\delta}{2}}( \Bar{\Omega} \times [0, \infty)) \times
C^{2+\delta,1+\frac{\delta}{2}}( \Bar{\Omega} \times [0, \infty))\times 
C^{2+\delta,\frac{\delta}{2}}( \Bar{\Omega} \times [0, \infty))
\end{equation*}
such that $0\leq u, v, w, z \in L^\infty(\Omega \times (0,\infty)).$ 
\end{theorem}
\begin{theorem}\label{theoremlocalParab}
For some $\delta\in(0,1)$ and $n\in \N$, let $\Omega \subset \R^n$ be a bounded domain of class $C^{2+\delta}$, $\tau=1$, $\chi, \xi, a_1,a_2, c_1, c_2, h_1, h_2 >0$ and $\alpha \in [0,1]$. Then there exists $C_{\partial \Omega}=C_{\partial \Omega}(\Omega, n)$ such that for every $u_0, v_0, w_0, z_0: \bar{\Omega}  \rightarrow \mathbb{R}^+$, with $u_0, v_0, w_0, z_0 \in C^{2+\delta}(\bar\Omega)$ complying with $u_{0\nu}=(\alpha-1)\chi h_1 u_0 v_0$,  $v_{0\nu}=-h_1 v_0$, 
$w_{0\nu}=(\alpha-1)\xi h_2 w_0z_0$ and $z_{0\nu}=- h_2 z_0$ on $\partial \Omega$ the following conclusion holds true:
whenever 
\begin{equation}\label{Estimate_bP}
b_1> c_1 + \frac{(\xi \alpha)^2 h_2 C_{\partial \Omega}}{16 c_2}, \quad   
b_2 > c_2 + \frac{(\chi \alpha)^2 h_1 C_{\partial \Omega}}{16 c_1} \quad \textrm{(being $\alpha>0$}),  
\end{equation}
problem \eqref{problem} admits a unique solution
\begin{equation*}
(u,v, w, z)\in (C^{2+\delta,1+\frac{\delta}{2}}( \Bar{\Omega} \times [0, \infty)))^4
\end{equation*}
such that $0\leq u, v, w, z \in L^\infty(\Omega \times (0,\infty)).$ 
\end{theorem}
\begin{remark}[Positive total flux and known results on single-species models]\label{RemakFlussiPositiviSingleSpecies}
As it is well known, for chemotaxis systems of form \eqref{problem}, the vectorial total flux associated with the species \(u\) is given by the vector field
$
\nabla u - \chi u \nabla v$ on $\partial\Omega \times (0,T_{\max})$.
Similarly, the vectorial total flux associated with the second species \(w\) is
$\nabla w - \xi w \nabla z.$ By virtue of the boundary conditions satisfied by \(u\), \(v\), \(w\), and \(z\), it follows that the total fluxes have the following expressions: 
\begin{equation}\label{toalFluxesExpression}
\bigl(\nabla u - \chi u \nabla v\bigr)\cdot \nu
    = \alpha h_1 \chi uv,
\qquad
\bigl(\nabla w - \xi w \nabla z\bigr)\cdot \nu
    = \alpha h_2 \xi wz.
\end{equation}
Hence, the resulting total fluxes are quadratic at the boundary, being essentially proportional to the bilinear forms generated by the interaction between the population densities and the corresponding chemical signals.

In this context, restricting our attention to the single-species setting, i.e. to the model
\begin{equation}\label{problemOneSpecies}
\begin{cases}
u_t =\Delta u - \chi \nabla\cdot(u\nabla v)+a_1u-b_1u^{2} -c_1|\nabla u|^{2} & \textrm{in } \Omega \times (0,\TM),\\  
\tau v_t= \Delta v - v+u &\textrm{in } \Omega \times (0,\TM),\\
 u(x,0) =u_{0}(x), \tau v(x,0) =\tau v_{0}(x)& x \in \bar{\Omega},
\end{cases}
\end{equation}
we emphasize the following recent results which, to the best of our knowledge, are the only ones available in the literature that allow for positive total fluxes through the boundary: 
\begin{itemize}
\item [$\triangleright$] For $c_1=0$, and  
under nonlinear Neumann boundary conditions leading to $u_{\nu} - \chi u v_{\nu}  = |u|^p \quad \text{with } p>1$, in \cite{Le2024} it is established that if \(p < \frac{3}{2}\) and \(b_1 > 0, n = 2\), or \(b_1\) is sufficiently large when \(n \ge 3\), then the parabolic-elliptic chemotaxis system in \eqref{problemOneSpecies} admits a unique nonnegative global-in-time classical solution that is bounded in \(\Omega \times (0, \infty)\).  A similar result is also true if \(p < \frac{3}{2}, n = 2\), and \(b_1 > 0\) or \(p < \frac{7}{5}, n = 3\), and \(b_1\) is sufficiently large for the parabolic-parabolic chemotaxis system.
\item [$\triangleright$] For both $\tau=0$ and $\tau=1$ and for any $n$, if nonlinear Neumann boundary conditions yielding $u_{\nu} - \chi u v_{\nu} = \chi \alpha h_1 \, u v$ (with \(\alpha > 0\)) are considered in \eqref{problemOneSpecies}, then in \cite{BAG2026} global and bounded solutions are obtained as long as $b_1$ and $c_1$ in the gradient-dependent logistic sources are sufficiently large.
\end{itemize}
We observe the behavior of the two total fluxes: one, \(|u|^p\), is subquadratic, while the other, \(uv\), is quadratic. Consistently with the intuitive argument, in order to balance a stronger inward flux, it is necessary to introduce a logistic with stronger damping effects. 

Definitely, and exactly as in the work \cite{BAG2026} (to which we refer the reader for a detailed discussion of the motivations and justification underlying the definition of the flow employed there and adapted here to the two-species setting), the interplay between the quadratic inflow and the dampening gradient term will also play a crucial role in the present study.
\end{remark}
\begin{remark}[Some considerations and challenging open problems]\label{ConsiderationOpenProbl}
We now present some remarks concerning the following model for $\tau_1,\tau_2 \in \{0,1\}$: 
\begin{equation}\label{problemRemark}
\begin{cases}
u_t =\Delta u - \chi \nabla\cdot(u\nabla v)+a_1u-b_1u^{2} -c_1|\nabla u|^{2} & \textrm{in } \Omega \times (0,\TM),\\  
\tau_1 v_t= \Delta v - v+w &\textrm{in } \Omega \times (0,\TM),\\
w_{t} =\Delta w- \xi \nabla\cdot(w\nabla z)+a_2w-b_2w^{2} -c_2|\nabla w|^{2} &\textrm{in } \Omega \times (0,\TM),\\ 
 \tau_2 z_t=\Delta z -z+u &\textrm{in } \Omega \times (0,\TM).\\
\end{cases}
\end{equation}
\begin{itemize}
\item [$\triangleright$] \textsf{On the validity of Theorem \ref{theoremlocalParab}}: The result of Theorem \ref{theoremlocalParab} remains valid in the cases $\tau_1 = 0$ and $\tau_2 = 1$, and vice versa for model \eqref{problemRemark}. (Further details will be given later, exactly in Remark \ref{Details} on page \pageref{page}.)
\item [$\triangleright$] \textsf{The case of the zero-flux, i.e. $\alpha=0$}: We point out that the results obtained in this paper are valid also in the case of zero total flux, i.e. $\alpha=0$ in \eqref{toalFluxesExpression}. Indeed, if we set $\tau_1=\tau_2=0$ in model \eqref{problemRemark}, this is immediately evident by simply inspecting condition \eqref{Estimate_Bag0} in Theorem \ref{theoremlocal}. In the other cases, where at least one of the parameters $\tau_i$ equals $1$, this conclusion cannot be directly inferred from condition \eqref{Estimate_bP}, but it follows instead from adapted  results established in \cite{IshidaLankeitVigliloro-Gradient}.
\item [$\triangleright$] \textsf{Some open problems for model \eqref{problemRemark}}: As the reader may have noticed, in the single-species setting the fully parabolic model \eqref{problem} is known to admit globally bounded solutions under two distinct sets of assumptions; see \cite[Theorem 5.1]{BAG2026}, for $\tau=1$. By contrast, for the fully parabolic two-species system, corresponding to $\tau_1=\tau_2=1$ in \eqref{problemRemark}, the analytical approach developed therein does not seem to carry over to the present setting. The main obstruction appears to arise from the mixed nonlinear terms involving both $u$ and $w$, whose treatment requires substantially finer estimates. On the other hand, although some computations simplify when one of the equations is elliptic, this fundamental difficulty persists in the mixed parabolic--elliptic cases, corresponding to either $\tau_1=0,\ \tau_2=1$ or $\tau_1=1,\ \tau_2=0$ in problem \eqref{problemRemark}. This indicates that the model studied here is not simply a direct extension of the framework developed in \cite{BAG2026}, but instead it presents nontrivial technical challenges. 

In contrast, no such obstruction arises in the doubly elliptic case $\tau_1=\tau_2=0$, corresponding to the model in Theorem \ref{theoremlocal}. Consequently, the two alternative conditions ensuring global boundedness established therein are consistent with those obtained in \cite{BAG2026} for the single-species model.
\end{itemize}
\end{remark}

\begin{remark}[On the role of $v_0$ and $z_0$ in Theorem \ref{theoremlocal}]
For the case $\tau=0$ in model \eqref{problem}, $v$ is uniquely determined by $w$ at each time $t$ (similarly $z$ by $u$). In particular, the initial datum $v_0$ is not prescribed independently, but it is defined as $v_0 := v(\cdot,0)$, where $v_0$ is the solution to
\[
\begin{cases}
\Delta v_0 - v_0 + w_0 = 0 & \text{in } \Omega, \\
\partial_{\nu} v_0 = - h_1 v_0 & \text{on } \partial\Omega,
\end{cases}
\quad \left(\textrm{similarly\;} z_0 := z(\cdot,0) \quad \textrm{is the solution to}\quad 
\begin{cases}
\Delta z_0 - z_0 + u_0 = 0 & \text{in } \Omega, \\
\partial_{\nu} z_0 = - h_2 z_0 & \text{on } \partial\Omega,
\end{cases}\right).
\]
\end{remark} 
\section{Preliminary: elliptic and parabolic regularity, and trace estimates}
We will make use of the following tools, dealing with regularity estimates and trace embedding results.
(Throughout the paper, unless otherwise specified, all constants $d_i$, $i=1, 2, \cdots$ are assumed to be positive.)
\begin{lemma}[Elliptic regularity]\label{RegElliptic}
For some $\delta\in(0,1)$ and $n\in \N$, let $\Omega \subset \R^n$ be a bounded domain of class $C^{2+\delta}$, $p>n$ and $h>0. $ If $0\leq \psi\in C^\delta (\overline\Omega)$, then the solution $\phi\in C^{2+\delta} (\overline\Omega)$ of problem 
\[
\begin{cases}
-\Delta \phi + \phi =\psi &\textrm{in } \Omega,\\
\phi_\nu=-h \phi &\textrm{on } \partial\Omega,
\end{cases}
\]
is nonnegative in $\bar{\Omega}$, and moreover it has the property that for every $q>2$ there is $C_E=C_E(n,\Omega,q,h)>0$ for which $\phi$ is such that 
\begin{equation}\label{EstReg}
\|\phi\|_{W^{1,q}(\Omega)}\leq C_E \|\psi\|_{L^q(\Omega)}.
\end{equation}
In addition, if $\psi\in L^p(\Omega),$ then 
\begin{equation}\label{Gradient}
\phi \in W^{1,\infty}(\Omega).
\end{equation}
\begin{proof}
We refer the reader to \cite[Lemma 6.2]{BAG2026}.
\end{proof}
\end{lemma}
\begin{lemma}[Parabolic regularity]\label{RegParab}
For some $\delta\in(0,1)$ and $n\in \N$, let $\Omega \subset \R^n$ be a bounded domain of class $C^{2+\delta}$, $T\in (0,\infty]$, $p>n$ and $h>0$. 
If $0\leq \psi\in C^{\delta,\frac{\delta}{2}}(\bar{\Omega}\times [0,T])$ and $\phi_0\in C^{2+\delta}(\Omega)$, with $\phi_{0\nu} = -h\phi_0$ on $\partial\Omega$, then the solution  $\phi\in C^{2+\delta,1+\frac{\delta}{2}}(\bar{\Omega}\times [0,T])$ of the problem
\begin{equation*}
\begin{cases}
    \phi_t= \Delta \phi - \phi  +\psi & \text{in } \Omega \times (0,T),\\
     \phi_\nu = -h \phi & \text{on } \partial\Omega \times (0,T),\\
    \phi(\cdot,0)=\phi_0 & \text{on } \Omega,
\end{cases}
\end{equation*}
is nonnegative in $\bar{\Omega}\times [0,T]$, and additionally it has the property that for every $q>1$ there exists $C_P=C_P(n,\Omega,q,h)>0$ for which $\phi$ is such that
\begin{equation}\label{tau1}
\int_0^t e^s \int_\Omega |\Delta \phi(\cdot,s)|^q\,ds \leq C_{P}\left[1+\int_0^t e^s \int_\Omega |\psi(\cdot,s)|^q\,ds\right] \quad \text{for all } t\in(0,T).
\end{equation}
Additionally, if  $\psi\in L^\infty((0,T);L^p(\Omega))$ then
\begin{equation}\label{tau1extension}
\phi \in L^\infty((0,T);W^{1,\infty}(\Omega)).
\end{equation}
\begin{proof}
See \cite[Lemma 6.2]{BAG2026}.
\end{proof}
\end{lemma}
\begin{lemma}[Trace inequality]\label{TraceLemma}
For some $\delta\in(0,1)$ and $n\in \N$, let $\Omega \subset \R^n$ be a bounded domain of class $C^{2+\delta}$. Then there is $C_{\partial \Omega}=
C_{\partial \Omega}(n, \Omega)$ positive such that for all $0\leq \psi\in C^{1}(\bar\Omega)$ and $\mathfrak{p}\in \{1,2\}$ the  following holds:
\begin{equation}\label{TraceInequ}
\int_{\partial \Omega} \psi^{\mathfrak{p}}\leq C_{\partial \Omega}\left(\int_\Omega \psi^{\mathfrak{p}}  +\int_\Omega |\nabla\psi|^{\mathfrak{p}}\right).
\end{equation} 
\begin{proof}
The proof can be found in \cite[page 315]{BrezisBook} or \cite[Theorem 1, page 258]{Evans-2010-PDEs}.
\end{proof}
\end{lemma}
\section{Local existence, dichotomy and boundedness criterion through Hölder continuity} 
The first step is to guarantee the existence of a solution, at least for short times, to problem \eqref{problem}, under investigation. Although the proof of this result is usually considered standard in the literature, we choose to provide it here as well, in order to make the paper complete and self-contained.
\begin{lemma}[Local existence] \label{LocalExistenceLemma}
For some $\delta\in(0,1)$ and $n\in \N$, let $\Omega \subset \R^n$ be a bounded domain of class $C^{2+\delta}$, $\tau \in \{0,1\}$, $\chi, \xi, a_1, a_2, b_1, b_2, c_1, c_2, h_1, h_2 > 0$ and $\alpha \in [0,1]$. Then for every $u_0, \tau v_0, w_0, \tau z_0: \bar{\Omega}  \rightarrow \R^+$, with $u_0,  \tau v_0, w_0,  \tau z_0 \in C^{2+\delta}(\bar\Omega)$ complying with 
\begin{equation*}
u_{0\nu}=(\alpha-1)\chi h_1 u_0 v_0,  \quad \tau v_{0\nu}= - h_1 \tau v_0, \quad w_{0\nu}=(\alpha-1)\xi h_2 w_0 z_0 \quad \textrm{and} \quad  \tau z_{0\nu}=- h_2 \tau z_0 \;  \textrm{ on }\partial \Omega
\end{equation*}
there exist $\TM\in (0,\infty]$ and a unique quadruple of functions $(u,v, w, z)$, with 
\begin{equation*}
(u,v, w, z)\in C^{2+\delta,1+\frac{\delta}{2}}( \Bar{\Omega} \times [0, \TM))\times C^{2+\delta, \tau + \frac{\delta}{2}}( \Bar{\Omega} \times [0, \TM)) \times C^{2+\delta,1+\frac{\delta}{2}}( \Bar{\Omega} \times [0, \TM))\times C^{2+\delta, \tau+\frac{\delta}{2}}( \Bar{\Omega} \times [0, \TM)),
\end{equation*}
solving problem \eqref{problem} and nonnegative in $\bar{\Omega}\times [0,\TM)$. Additionally,
 \begin{equation}\label{dictomyCriteC2+del} 
 \text{if} \quad \TM<\infty \quad \text{then} \quad \limsup_{t \to \TM} \left(\|u(\cdot,t)\|_{C^{2+\delta}(\bar\Omega)}+\|v(\cdot,t)\|_{C^{2+\delta}(\bar\Omega)}+\|w(\cdot,t)\|_{C^{2+\delta}(\bar\Omega)}+\|z(\cdot,t)\|_{C^{2+\delta}(\bar\Omega)}\right)=\infty.
 \end{equation}
\begin{proof}
Let us begin with the question of existence. For any $R > 0$, and for some $0 < T \leq 1$ which will be be specified below, we consider the closed, bounded, and convex subset of $C^{\delta,\frac{\delta}{2}}(\bar{\Omega}\times [0,T])$
$$S_{T}=\{(u, w) \in (C^{\delta,\frac{\delta}{2}}(\bar{\Omega}\times [0,T]))^2: u, w \geq 0,  \, \|u(\cdot,t)-u_0\|_{C^{\delta}(\bar{\Omega})}\le R,\, 
\|w(\cdot,t)-w_0\|_{C^{\delta}(\bar{\Omega})}\le R,\, \; \text{for\ all}\
t\in[0,T]\}.$$
Given an element $(\tilde u,\tilde w)\in S_T$, Lemma \ref{RegElliptic} and Lemma \ref{RegParab}
guarantee the existence of solutions $v$ and $z$ to problems
\begin{equation*}\label{2.2}
\begin{cases}
\tau v_t -\Delta v+ v=\tilde{w}&\text{in}\ \Omega\times(0,T),\\
v_\nu=-h_1 v&\text{on}\ \partial\Omega\times(0,T),\\
\tau v(x,0) = \tau v_0(x) & x \in \bar{\Omega}
\end{cases}
\end{equation*}
and
\begin{equation*}\label{2.2z}
\begin{cases}
\tau z_t -\Delta z+ z=\tilde{u}&\text{in}\ \Omega\times(0,T),\\
z_\nu=-h_2 z&\text{on}\ \partial\Omega\times(0,T),\\
\tau z(x,0) = \tau z_0(x) & x \in \bar{\Omega}
\end{cases}
\end{equation*}
respectively.
By $\tau v_0, \tau z_0 \in C_\nu^{2+\delta}(\bar{\Omega})$ and  elliptic and parabolic regularity results for oblique derivative problems (\cite[Theorem 6.30]{GilbarTrudinger} for $\tau=0$, and \cite[Corollary 5.1.22]{LunardiBook}, \cite[Theorem IV.5.3]{LSUBookInequality} for $\tau=1$),  we have that 
\begin{equation*}
v, z \in C^{2+\delta, \tau+\frac{\delta}{2}}(\bar\Omega \times [0,T])
\quad \textrm{and}\quad \sup_{t\in [0,T]} \left(\lVert v(\cdot, t)\rVert_{C^{2+\delta}(\bar\Omega)} + \lVert z(\cdot, t)\rVert_{C^{2+\delta}(\bar\Omega)}\right) 
\leq H
\end{equation*}
with $H=H(R)>0$. In light of the properties of $v, \nabla v$ and $z, \nabla z$ on $\bar{\Omega}\times [0,T]$, problems
\begin{equation}\label{2.1}
	\begin{cases}
		u_t=\Delta u -\chi  \nabla u  \cdot \nabla v-\chi u \Delta v +  a_1 u-b_1 u^2 - c_1|\nabla u|^2
		&\text{in}\ \Omega\times(0,T),\\
		u_\nu=(\alpha-1)h_1 \chi u v&\text{on}\ \partial\Omega\times(0,T),\\
		u(x,0)=u_0(x)&x\in\overline{\Omega},
	\end{cases}
\end{equation}
and
\begin{equation}\label{2.1w}
	\begin{cases}
		w_t=\Delta w -\xi  \nabla w  \cdot \nabla z-\xi w \Delta z +  a_2 w-b_2 w^2 - c_2|\nabla w|^2
		&\text{in}\ \Omega\times(0,T),\\
		w_\nu=(\alpha-1)h_2 \xi w z&\text{on}\ \partial\Omega\times(0,T),\\
		w(x,0)=w_0(x)&x\in\overline{\Omega},
	\end{cases}
\end{equation}
can be rewritten as  general systems of the form
\begin{equation*}
\begin{cases}
u_t=\mathcal{A}u +\varphi(x,t,u,\nabla u) &\text{in}\ \Omega\times(0,T),\\
\mathcal{B}_1 u=u_\nu-(\alpha-1)h_1 \chi u v=0&\text{on}\ \partial\Omega\times(0,T),\\
u(x,0)=u_0(x)&x\in\overline{\Omega}
\end{cases}
\qquad \textrm{and} \qquad 
\begin{cases}
w_t=\mathcal{A}w +\varphi(x,t,w,\nabla w) &\text{in}\ \Omega\times(0,T),\\
\mathcal{B}_1 w=w_\nu-(\alpha-1)h_2 \xi w z=0&\text{on}\ \partial\Omega\times(0,T),\\
w(x,0)=w_0(x)&x\in\overline{\Omega}.
\end{cases}
\end{equation*}
Since the above problems satisfy \cite[(7.3.5)]{LunardiBook}, by recalling that $u_0, w_0 \in C^{2+\delta}_\nu(\bar\Omega)$,  \cite[Proposition 7.3.3]{LunardiBook} and \cite[Theorem V 6.1]{LSUBookInequality} (see also \cite[Theorem 8.5.4]{LunardiBook}) ensure the existence of some $0<T< 1$, as above, such that problems \eqref{2.1}
and \eqref{2.1w} have a unique solution 
\begin{equation*}
u, w \in C^{2+\delta,1+\frac{\delta}{2}}(\bar\Omega\times [0,T]).
\end{equation*} 
Moreover, there exists $K_1=K_1(R)>0$ such that 
\[
\|u\|_{C^{\delta,\frac{\delta}{2}}(\bar\Omega\times [0,T])} + \|w\|_{C^{\delta,\frac{\delta}{2}}(\bar\Omega\times [0,T])} \le K_1(R),
\]
and therefore for proper $K_2, K>0$ we have
\[
\|u(\cdot,t)-u_0\|_{C^\delta(\bar\Omega)}
+
\|w(\cdot,t)-w_0\|_{C^\delta(\bar\Omega)} \leq K_2 T^{\frac{\delta}{2}}(\|u\|_{C^{\delta,\frac{\delta}{2}}(\bar\Omega\times [0,T])} + \|w\|_{C^{\delta,\frac{\delta}{2}}(\bar\Omega\times [0,T])} ) \le K T^{\frac{\delta}{2}} \quad \textrm{for all } t \in [0,T].
\]
Choosing
\[
T\le\left(\frac{R}{K}\right)^{\frac{2}{\delta}},
\]
we derive 
\[
\|u(\cdot,t)-u_0\|_{C^\delta(\bar\Omega)}
+
\|w(\cdot,t)-w_0\|_{C^\delta(\bar\Omega)}
\le R \quad \textrm{for all } t \in [0,T].
\]
Furthermore, since $0$ is a subsolution of problems \eqref{2.1} and \eqref{2.1w}, the nonnegativity of $u$ and $w$ on $\Omega \times (0,T)$ follows from the parabolic comparison principle (see \cite[Theorem 8.2]{CrandallEtAl_User1992} supported by \cite[Theorem 1]{Friedman1958}), while the nonnegativity of $v$ and $z$ is a direct consequence of the elliptic and parabolic maximum principle.
Henceforth, the above reasoning implies that $\Phi(S_T)\subset S_T$, where $\Phi(\tilde u,\tilde w)=(u,w)$ with $u$ and $w$ solutions of problems \eqref{2.1} and 
\eqref{2.1w}. 
Moreover, $\Phi$ is bounded in $C^{2+\delta,1+\frac{\delta}{2}}(\bar{\Omega} \times [0,T])$, and compact in $C^{\delta,\frac{\delta}{2}}(\bar{\Omega} \times [0,T])$, as the \cite[Ascoli--Arzel\`a Theorem 4.25]{BrezisBook} ensure the natural embedding of $C^{2+\delta,\tau+\frac{\delta}{2}}(\bar{\Omega}\times [0,T]))$ into $C^{\delta,\frac{\delta}{2}}(\bar{\Omega}\times [0,T]))$. Furthermore, using the continuous dependence of the solutions defining $\Phi$ on the initial data, we deduce its continuity. Therefore, Schauder's fixed point theorem establishes the existence of a fixed point $(u,w) \in S_T$ for $\Phi$. 

Concerning uniqueness, the result can be obtained by adapting the proof in \cite[Lemma 7.1]{BAG2026}; we shall here present only the steps that are different and of primary significance. Let $(u_1,v_1, w_1, z_1)$ and $(u_2,v_2, w_2, z_2)$ be two different nonnegative classical solutions of problem \eqref{problem} in $\Omega\times (0,T)$ with the same initial data $u_1(\cdot,0)=u_2(\cdot,0)=u_0(x)$, $\tau v_1(\cdot,0)=\tau v_2(\cdot,0)= \tau v_0(x)$, $w_1(\cdot,0)=w_2(\cdot,0)=w_0(x)$
and $\tau z_1(\cdot,0)= \tau z_2(\cdot,0)=\tau z_0(x)$. 
We restrict our analysis to the case $\tau=0$ (the case $\tau=1$ being analogous, as will be discussed later). For $i=1,2$ let us consider such problems:
\begin{equation}\label{3.2_Bis}
\begin{cases}
-\Delta v_i+ v_i=w_i&\text{in}\ \Omega\times(0,T),\\
(v_i)_\nu=-h_1 v_i&\text{on}\ \partial\Omega\times(0,T),
\end{cases}
\end{equation}
\begin{equation}\label{3.1_BisA}
\begin{cases}
(u_i)_t=\Delta u_i-\chi \nabla \cdot (u_i \nabla v_i)+a_1 u_i-b_1 u_i^2-c_1 |\nabla u_i|^2 
 & \text{in}\ \Omega\times(0,T),\\
(u_i)_\nu=(\alpha-1)\chi h_1 u_i v_i&\text{on}\ \partial\Omega\times(0,T),\\
u_i(x,0)=u_0(x)&x\in\overline{\Omega},
\end{cases}
\end{equation}
\begin{equation}\label{3.2_Bis_z}
\begin{cases}
-\Delta z_i+ z_i=u_i&\text{in}\ \Omega\times(0,T),\\
(z_i)_\nu=-h_2 z_i&\text{on}\ \partial\Omega\times(0,T),
\end{cases}
\end{equation}
and
\begin{equation}\label{3.1_BisA_w}
\begin{cases}
(w_i)_t=\Delta w_i-\xi \nabla \cdot (w_i \nabla z_i)+a_2 w_i-b_2 w_i^2-c_2 |\nabla w_i|^2 
 & \text{in}\ \Omega\times(0,T),\\
(w_i)_\nu=(\alpha-1)\xi h_2 w_i z_i&\text{on}\ \partial\Omega\times(0,T),\\
w_i(x,0)=w_0(x)&x\in\overline{\Omega}.
\end{cases}
\end{equation}
Owing to the $C^2$-regularity of $u_1, u_2, v_1, v_2, w_1, w_2$ and $z_1, z_2$, we may define
\begin{equation}\label{ConstantsC1-2-3} 
\begin{split}
	C_1=
		 & \max\{\|u_1\|_{L^\infty(\Omega\times(0,T))},
	\|u_2\|_{L^\infty(\Omega\times(0,T))}, \lVert v_1\rVert_{L^\infty(\Omega \times (0,T))}, \lVert v_2\rVert_{L^\infty(\Omega \times (0,T))}, \\ & \quad \quad  \lVert \nabla u_1\rVert_{L^\infty(\Omega \times (0,T))},\lVert \nabla u_2\rVert_{L^\infty(\Omega \times (0,T))}, \lVert \nabla v_1\rVert_{L^\infty(\Omega \times (0,T))}\}
	\end{split}
\end{equation}
and 
\begin{equation}\label{ConstantsC1-2-3w} 
\begin{split}
	C_2=
		 & \max\{\|w_1\|_{L^\infty(\Omega\times(0,T))},
	\|w_2\|_{L^\infty(\Omega\times(0,T))}, \lVert z_1\rVert_{L^\infty(\Omega \times (0,T))}, \lVert z_2\rVert_{L^\infty(\Omega \times (0,T))}, \\ & \quad \quad  \lVert \nabla w_1\rVert_{L^\infty(\Omega \times (0,T))},\lVert \nabla w_2\rVert_{L^\infty(\Omega \times (0,T))}, \lVert \nabla z_1\rVert_{L^\infty(\Omega \times (0,T))}\}.
	\end{split}
\end{equation}
Thereupon, the Mean Value Theorem yields
\begin{equation}\label{MeanValueTheroem}  
|u_1^2-u_2^2| \leq 2 C_1|u_1-u_2| \quad \textrm{ and } \quad ||\nabla u_1|^2-|\nabla u_2|^2|\leq 2 C_1||\nabla u_1|-|\nabla u_2|| \quad \text{ in}\ \Omega\times(0,T)
\end{equation} 
and
\begin{equation}\label{MeanValueTheroemw}  
|w_1^2-w_2^2| \leq 2 C_2|w_1-w_2| \quad \textrm{ and } \quad ||\nabla w_1|^2-|\nabla w_2|^2|\leq 2 C_2||\nabla w_1|-|\nabla w_2|| \quad \text{ in}\ \Omega\times(0,T).
\end{equation} 
Within this framework, examining problems \eqref{3.2_Bis} and  \eqref{3.2_Bis_z}, it follows that $V:=v_1-v_2$ and $Z:=z_1-z_2$ satisfy respectively
\begin{equation}\label{ProblElliptic}
\begin{cases}
-\Delta V+V=W:=w_1-w_2&\text{in}\ \Omega\times(0,T),\\
V_\nu=-h_1 V & \text{on}\ \partial\Omega\times(0,T)
\end{cases}
\end{equation}
and 
\begin{equation}\label{ProblElliptic_Z}
\begin{cases}
-\Delta Z+Z=U:=u_1-u_2&\text{in}\ \Omega\times(0,T),\\
Z_\nu=-h_2 Z & \text{on}\ \partial\Omega\times(0,T).
\end{cases}
\end{equation}
Accordingly, by Young's inequality, testing equation \eqref{ProblElliptic} with $V$ and equation \eqref{ProblElliptic_Z} with $Z$ gives the estimates
\begin{equation}\label{3.4}
\int_\Omega|\nabla V|^2+\int_\Omega V^2\le \const{Giu-2Bis}\int_\Omega W^2\quad \textrm{for all }\, t\in(0,T)
\end{equation}
and
\begin{equation}\label{3.4_Z}
\int_\Omega|\nabla Z|^2+\int_\Omega Z^2\le \const{Giu-2Biss}\int_\Omega U^2\quad \textrm{for all }\, t\in(0,T).
\end{equation}
Conversely, from \eqref{3.1_BisA} and \eqref{3.1_BisA_w} some straightforward computations infer that $U=u_1-u_2$ and $W=w_1-w_2$ fulfill 
\begin{equation}\label{3.1_Bis}
	\begin{cases}
  U_t=\Delta U
		-\chi \nabla \cdot (U \nabla v_1+u_2\nabla V)
		+a_1 U -b_1 (u_1^2-u_2^2)-c_1(|\nabla u_1|^2-|\nabla u_2|^2) 
 & \text{in}\ \Omega\times(0,T),\\
U_\nu=(\alpha-1)\chi h_1 \left( u_1 v_1-u_2v_2\right)&\text{on}\ \partial\Omega\times(0,T),
  \\
U(x,0)=0&x\in\overline{\Omega}
\end{cases}
\end{equation}
and 
\begin{equation}\label{3.1_Bis_W}
\begin{cases}
  W_t=\Delta W
		-\xi \nabla \cdot (W \nabla z_1+w_2\nabla Z)
		+a_2 W -b_2 (w_1^2-w_2^2)-c_2(|\nabla w_1|^2-|\nabla w_2|^2) 
 & \text{in}\ \Omega\times(0,T),\\
W_\nu=(\alpha-1)\xi h_2\left( w_1 z_1-w_2 z_2\right)&\text{on}\ \partial\Omega\times(0,T),
  \\
W(x,0)=0&x\in\overline{\Omega}.
\end{cases}
\end{equation}
In this manner, by exploiting \eqref{3.2_Bis}--\eqref{ProblElliptic_Z} and the inequalities 
\[
|\nabla u_1|-|\nabla u_2|\leq |\nabla (u_1-u_2)|=|\nabla U| \textrm{ and } |\nabla w_1|-|\nabla w_2|\leq |\nabla (w_1-w_2)|=|\nabla W|,\] problems \eqref{3.1_Bis} and \eqref{3.1_Bis_W} yield for all $t\in (0,T)$
\begin{equation}\label{MainForUniqueness}  
\begin{split}
\frac{1}{2}\frac{d}{dt}\int_\Omega U^2+\int_\Omega |\nabla U|^2&\leq C_3\int_{\partial \Omega}U^2+C_4 \int_{\partial \Omega}|U  V| +C_5\int_\Omega |U \nabla U| +C_6\int_\Omega |\nabla U \cdot \nabla V| +C_7\int_\Omega  U^2
	\end{split}
\end{equation}
and
\begin{equation}\label{MainForUniquenessW}  
\begin{split}
\frac{1}{2}\frac{d}{dt}\int_\Omega W^2+\int_\Omega |\nabla W|^2&\leq C_8\int_{\partial \Omega}W^2+C_9 \int_{\partial \Omega}|W  Z| +C_{10} \int_\Omega |W \nabla W| +C_{11} \int_\Omega |\nabla W \cdot \nabla Z| +C_{12} \int_\Omega W^2.
	\end{split}
\end{equation}
Now, in order to control the integrals at the right hand side of \eqref{MainForUniqueness} and \eqref{MainForUniquenessW}, we apply the trace
inequality in \eqref{TraceInequ} (with $\mathfrak{p}=2$) in conjunction with Young's inequality, this yields  
\begin{equation}\label{IneqA-A}
C_3\int_{\partial \Omega}U^2+C_4 \int_{\partial \Omega}|U  V| +C_5\int_\Omega |U \nabla U| +C_6\int_\Omega |\nabla U \cdot \nabla V| \leq C_{13} \int_\Omega U^2 + \int_\Omega |\nabla U|^2 + C_{14} \int_\Omega V^2+ C_{15} \int_\Omega |\nabla V|^2
\end{equation}
and
\begin{equation}\label{IneqA-AW}
C_8\int_{\partial \Omega}W^2+C_9 \int_{\partial \Omega}|W  Z| +C_{10} \int_\Omega |W \nabla W| +C_{11} \int_\Omega |\nabla W \cdot \nabla Z| \leq C_{16} \int_\Omega W^2 + \int_\Omega |\nabla W|^2 + C_{17} \int_\Omega Z^2+ C_{18} \int_\Omega |\nabla Z|^2.
\end{equation}
Lastly, by adding \eqref{MainForUniqueness} and \eqref{MainForUniquenessW} and inserting \eqref{IneqA-A} and \eqref{IneqA-AW} in such expression and invoking relations \eqref{3.4} and \eqref{3.4_Z}, we deduce for a proper $\tilde{C}=\tilde{C}(T,C_{\partial \Omega})>0$ that 
\[	
\frac{d}{dt}\int_\Omega (U^2 + W^2) \le\tilde{C}\int_\Omega (U^2+W^2),\qquad t\in(0,T).
\]
By recalling that $U(x,0)=W(x,0)=0$ in $\bar{\Omega}$, we derive from the initial condition $\int_\Omega (U^2(x,0)+W^2(x,0))\,dx=0$ that $U=u_1-u_2=0$ and $W=w_1-w_2=0$ on $\Omega \times (0,T)$ and  as a consequence, from problems \eqref{3.2_Bis} and \eqref{3.2_Bis_z}, we obtain $v_1-v_2=0$ and $z_1-z_2=0$ on $\Omega \times (0,T)$; achieving the claim. (For $\tau = 1$, the proof follows a similar line of argument, but the reasoning relies on the evolutive behavior of $\frac{1}{2} \int_\Omega (U^2+V^2+W^2+Z^2)$.)

Standard prolongation techniques ensure the existence of a maximal time of existence
$\TM \in (0,\infty]$, in the sense that either $\TM=\infty$, or if $\TM<\infty$ no solution belonging to $C^{2+\delta,\tau+\frac{\delta}{2}}(\bar{\Omega}\times [0,\TM])$ may exist and henceforth relation \eqref{dictomyCriteC2+del} has to be fulfilled. For more details, we refer to \cite[Lemma 7.1]{BAG2026}.
\end{proof}
\end{lemma}
From this point onward, $(u, v, w, z)$ denotes the unique local solution of model \eqref{problem} on $\Omega \times (0,\TM)$ as established in Lemma \ref{LocalExistenceLemma}, for which relation \eqref{dictomyCriteC2+del} holds; in this case, the solution blows up at 
$\TM$ in the $C^{2+\delta}(\bar{\Omega})$-norm. It is worth emphasizing that the blow-up of the solution in the $C^{2+\delta}$ norm does not, in principle, imply that the solution itself becomes unbounded. However, the extension criterion stated below (which we will exploit with $\psi=u$, $f=-\chi \nabla v$, $a=a_1$, $b=b_1$, $c=c_1$ and with $\psi=w$, $f=-\xi \nabla z$, $a=a_2$, $b=b_2$, $c=c_2$) rules out this possibility. In particular, due to the quadratic growth of the logistic gradient term, any bounded solution is actually Hölder continuous. 
\begin{lemma}[Extension and boundedness criteria]\label{LemmaMoserType}
For some $\delta\in(0,1)$ and $n\in \N$, let $\Omega \subset \R^n$ be a bounded domain of class $C^{2+\delta}$, $\TM\in (0,\infty]$, $a,b,c > 0$, $f=f(x,t) \in \left(C^1(\bar\Omega\times (0,\TM))\cap L^\infty((0,\TM);L^\infty(\Omega))\right)^n$, with $f\cdot \nu\geq 0$ for all $(x,t)\in \partial \Omega \times (0,\TM)$, and 
$0\leq g=g(x,t) \in  C^0(\bar\Omega\times [0,\TM))$. If $0 \leq \psi\in C^{2+\delta,1+\frac{\delta}{2}}(\bar\Omega \times [0,\TM))\cap L^\infty((0,\TM);L^1(\Omega))$ solves
\begin{equation*}
\begin{cases}
\psi_t= \Delta \psi +\nabla \cdot (f \psi)+a\psi-b\psi^2-c|\nabla \psi|^2 & \textrm{in }\; \Omega \times (0,T)\\
\psi_\nu +g\psi= 0 & \textrm{on }\; \partial \Omega \times (0,\TM),\\
\end{cases}
\end{equation*}
and satisfies $\psi_\nu(x,0)+g(x,0)\psi(x,0)=0$ for all $x\in \partial \Omega$; then we have $\TM=\infty$ and in particular $\psi \in L^\infty((0,\infty);L^\infty(\Omega))$.
\begin{proof}
This result is consequence of \cite[Lemma 6.5]{BAG2026} which gives that $\psi\in L^\infty((0,\TM);C^{1+\delta}(\bar{\Omega}))$ and, in turn \cite[Lemma 7.2]{BAG2026} gives the claim invoking regularity theories.
\end{proof}
\end{lemma}
\begin{remark}\label{RemarkApplicabilityExtension}
We observe that, in essence, the extension argument discussed above only requires uniform-in-time bounds for the $L^1$-norm of $\psi$ and the $L^\infty$-norm of $f$. In our setting, $f=-\chi\nabla v$, and its regularity depends on that of $u$ through the equation satisfied by $v$ (an analogous argument applies to $w$ and $z$). In particular, in order to guarantee the uniform boundedness of $f$, it is sufficient to establish the uniform boundedness of $u$ in $L^p$, for $p$ sufficiently large.
\end{remark}
\section{A priori estimates}
Accordingly to Remark \ref{RemarkApplicabilityExtension}, 
we here concentrate on deriving a priori integral bounds of $u$ and $w$ in some $L^p(\Omega))$ for large $p>1$, whose proof relies essentially on the uniform-in-time boundedness of $\int_\Omega u$ and $\int_\Omega w$ on $(0,\TM)$.
\subsection{Temporal uniform bound on the mass for $\tau=0$}
We will study the behavior of the functional $\int_\Omega (u+w)$, establishing that it remains bounded for all $t \in (0,\TM)$.
\begin{lemma}\label{BoundednessMass}
Let $\chi, \xi, a_1, a_2, c_1, c_2, h_1, h_2 >0$, $\alpha \in [0,1]$, $C_{\partial \Omega}$ be the trace constant in Lemma \ref{TraceLemma} and either 
\begin{equation}\label{Estimate_bLemma}
b_1> c_1 + \frac{(\xi \alpha)^2 h_2 C_{\partial \Omega}}{16 c_2}, \quad   
b_2 > c_2 + \frac{(\chi \alpha)^2 h_1 C_{\partial \Omega}}{16 c_1} \quad (\textrm{being } \alpha >0),  
\end{equation}
or 
\begin{equation}\label{BagTao0Lem}
\alpha \max\{\chi,\xi\} <\min\left\{4c_1, 4 c_2, \frac{2 b_1}{3},\frac{2 b_2}{3}\right\}.
\end{equation}
Then $u, w\in L^\infty((0,\TM);L^1(\Omega))$, in the sense that there exist $m_0, m_1 >0$ such that 
\begin{equation*}
\int_\Omega u \leq m_0 \quad \textrm{and} \quad \int_\Omega w \leq m_1 \quad \textrm{for all } t \in (0, \TM).
\end{equation*}
\begin{proof}
We begin by integrating the first equation of problem \eqref{problem} over the domain $\Omega$. Upon applying the divergence theorem, taking into account the Robin boundary conditions, and employing Young's inequality, we derive the following estimate 
\begin{equation}\label{estimate}
\begin{split}
\frac{d}{dt} \int_\Omega u &= \chi \alpha h_1 \int_{\partial \Omega} uv + a_1 \int_\Omega u - b_1 \int_\Omega u^2 - c_1 \int_\Omega  |\nabla u|^2\\
& \leq \frac{c_1}{C_{\partial \Omega}} \int_{\partial \Omega} u^2 
+ \frac{(\chi \alpha h_1)^2 C_{\partial \Omega}}{4 c_1}\int_{\partial \Omega} v^2 + a_1 \int_\Omega u - b_1 \int_\Omega u^2 - c_1 \int_\Omega  |\nabla u|^2 \quad \textrm{on } (0,\TM).
\end{split}
\end{equation}
On the other hand, testing the second equation in \eqref{problem} with $\frac{(\chi \alpha)^2 h_1 C_{\partial \Omega}}{4 c_1} v$, integrating over $\Omega$, and invoking Young's inequality (with $\epsilon_1>0$ specified later), yields
\begin{equation*}
\begin{split}
&\frac{(\chi \alpha)^2 h_1 C_{\partial \Omega}}{4 c_1}  \int_\Omega |\nabla v|^2 + \frac{(\chi \alpha h_1)^2 C_{\partial \Omega}}{4 c_1} \int_{\partial \Omega} v^2 + \frac{(\chi \alpha)^2 h_1 C_{\partial \Omega}}{4 c_1}   \int_\Omega v^2
= \frac{(\chi \alpha)^2 h_1 C_{\partial \Omega}}{4 c_1}  \int_\Omega wv \\
&\leq \epsilon_1 \int_\Omega w^2 + \left(\frac{(\chi \alpha)^2 h_1 C_{\partial \Omega}}{4 c_1}\right)^2\frac{1}{4 \epsilon_1}  \int_\Omega v^2 \quad \textrm{for all } t \in (0, T_{\max}),
\end{split}
\end{equation*}
which implies
\begin{equation*}
\frac{(\chi \alpha h_1)^2 C_{\partial \Omega}}{4 c_1}  \int_{\partial \Omega} v^2 \leq \epsilon_1 \int_\Omega w^2 - \frac{(\chi \alpha)^2 h_1 C_{\partial \Omega}}{4 c_1}
\left(1-\frac{(\chi \alpha)^2 h_1 C_{\partial \Omega}}{16 c_1 \epsilon_1}\right) \int_\Omega v^2 \quad \textrm{on } (0, T_{\max}).
\end{equation*}
Substituting the foregoing estimate into \eqref{estimate}, and applying the trace inequality to the term $\frac{c_1}{C_{\partial \Omega}} \int_{\partial \Omega} u^2$
as stated in Lemma \ref{TraceLemma} (with $\psi = u$, $\mathfrak{p} = 2$), we infer that
\begin{equation}\label{estimateMu}
\begin{split}
\frac{d}{dt} \int_\Omega u &\leq  a_1 \int_\Omega u + \left(c_1- b_1 \right) \int_\Omega u^2 + \epsilon_1 \int_\Omega w^2 - \frac{(\chi \alpha)^2 h_1 C_{\partial \Omega}}{4 c_1}\left(1-\frac{(\chi \alpha)^2 h_1 C_{\partial \Omega}}{16 c_1 \epsilon_1}\right) \int_\Omega v^2 \quad \textrm{for all } t \in (0,\TM).  
\end{split}
\end{equation}
Proceeding analogously with the third equation of problem \eqref{problem}, one arrives at the following estimate (with $\epsilon_2>0$ chosen later)
\begin{equation}\label{estimateMw}
\begin{split}
\frac{d}{dt} \int_\Omega w &\leq  a_2 \int_\Omega w + (c_2- b_2) \int_\Omega w^2 + \epsilon_2 \int_\Omega u^2 - \frac{(\xi \alpha)^2 h_2 C_{\partial \Omega}}{4 c_2}\left(1-\frac{(\xi \alpha)^2 h_2 C_{\partial \Omega}}{16 c_2 \epsilon_2}\right) \int_\Omega z^2 \quad \textrm{on } (0,\TM).  
\end{split}
\end{equation}
By summing bounds \eqref{estimateMu} and \eqref{estimateMw}, we obtain 
\begin{equation*}\label{sumMuw}
\begin{split}
&\frac{d}{dt} \left(\int_\Omega u + \int_\Omega w \right) \leq  a_1 \int_\Omega u + 
a_2 \int_\Omega w + (c_1+ \epsilon_2 - b_1) \int_\Omega u^2 + (c_2 + \epsilon_1- b_2)
\int_\Omega w^2 \\
&- \frac{(\chi \alpha)^2 h_1 C_{\partial \Omega}}{4 c_1}\left(1-\frac{(\chi \alpha)^2 h_1 C_{\partial \Omega}}{16 c_1 \epsilon_1}\right) \int_\Omega v^2
- \frac{(\xi \alpha)^2 h_2 C_{\partial \Omega}}{4 c_2}\left(1-\frac{(\xi \alpha)^2 h_2 C_{\partial \Omega}}{16 c_2 \epsilon_2}\right) \int_\Omega z^2 \quad \textrm{for all } t \in (0,\TM).
\end{split}
\end{equation*}
In light of assumption \eqref{Estimate_bLemma}, one can select $\epsilon_1$ and $\epsilon_2$ such that 
\begin{equation*}
\frac{(\chi \alpha)^2 h_1 C_{\partial \Omega}}{16 c_1} < \epsilon_1 < b_2-c_2
\quad \textrm{and} \quad
\frac{(\xi \alpha)^2 h_2 C_{\partial \Omega}}{16 c_2} < \epsilon_2 < b_1-c_1.
\end{equation*}
Consequently, a double application of Young's inequality gives constants $\const{yy1}$, 
$\const{yy2}$ and $\const{yy3}$ such that
\[
\frac{d}{dt} \left(\int_\Omega u + \int_\Omega w \right)  \leq -\const{yy1} \int_{\Omega} u 
-\const{yy2} \int_{\Omega} w + \const{yy3} \quad \text{on} \quad (0, T_{\max}),
\]
which, together with $\int_{\Omega} u(x,0)\,dx=\int_{\Omega} u_0(x)\,dx$ and $\int_{\Omega} w(x,0)\,dx=\int_{\Omega} w_0(x)\,dx$, leads to the statement.

On the other hand, the same objective can be attained through an alternative line of reasoning, namely by exploiting \eqref{BagTao0Lem}.
In particular, upon adding the term $\int_\Omega u$ to both sides of \eqref{estimate}, we arrive at
\begin{equation}\label{Bag00}
\frac{d}{dt} \int_\Omega u  + \int_\Omega u  = -c_1 \int_\Omega |\nabla u|^2 - b_1 \int_\Omega u^2  + (a_1 + 1) \int_\Omega u  + \chi \alpha h_1 \int_{\partial \Omega} uv \quad \textrm{on } (0,T_{\max}).
\end{equation}
Subsequently, for $\alpha \max\{\chi,\xi\} \leq \beta<\min\left\{4c_1, 4 c_2, \frac{2 b_1}{3},\frac{2 b_2}{3}\right\}$, we test the second equation in problem \eqref{problem} with $\beta (u + v)$ and integrate over $\Omega$; this yields 
\begin{equation}\label{Bag01}
\begin{split}
0 = \int_\Omega \beta (u + v) \left( \Delta v + w - v \right) =& -\beta \int_\Omega (\nabla u \cdot \nabla v)  - \beta \int_\Omega |\nabla v|^2 - \beta h_1 \int_{\partial \Omega} (uv + v^2)\\
& - \beta  \int_\Omega v^2 - \beta  \int_\Omega uv + \beta  \int_\Omega uw + \beta  \int_\Omega vw \quad \textrm{on } (0,T_{\max}).
\end{split}
\end{equation}
Combining \eqref{Bag00} and \eqref{Bag01} and neglecting non positive terms and
also in view of the constraint imposed on $\beta$ lead to
\begin{equation*}
\begin{split}
\frac{d}{dt} \int_\Omega u  + \int_\Omega u \leq & -c_1 \int_\Omega |\nabla u|^2  - \beta \int_\Omega (\nabla u \cdot \nabla v) - \beta \int_\Omega |\nabla v|^2 
+ (a_1 + 1) \int_\Omega u  + \beta  \int_\Omega uw \\
&+ \beta  \int_\Omega vw - b_1 \int_\Omega u^2 - \beta \int_\Omega v^2 \quad \textrm{for all }  t \in (0,T_{\max}).
\end{split}
\end{equation*}
Similarly to what we have done above, we infer
\begin{equation*}
\begin{split}
\frac{d}{dt} \int_\Omega w + \int_\Omega w \leq & -c_2 \int_\Omega |\nabla w|^2  - \beta \int_\Omega (\nabla w \cdot \nabla z) - \beta \int_\Omega |\nabla z|^2 
+ (a_2 + 1) \int_\Omega w  + \beta  \int_\Omega uw \\
&+ \beta  \int_\Omega uz - b_2 \int_\Omega w^2 - \beta \int_\Omega z^2 \quad \textrm{on } (0,T_{\max}).
\end{split}
\end{equation*}
By adding the previous inequalities, we derive for all $t \in (0,\TM)$
\begin{equation}\label{SumB}
\begin{split}
&\frac{d}{dt} \left(\int_\Omega u + \int_\Omega w\right) + \int_\Omega u  + \int_\Omega w \leq -c_1 \int_\Omega |\nabla u|^2  - \beta \int_\Omega (\nabla u \cdot \nabla v) 
- \beta \int_\Omega |\nabla v|^2  \\
&+ (a_1 + 1) \int_\Omega u  + 2 \beta  \int_\Omega uw + \beta  \int_\Omega vw - b_1 \int_\Omega u^2 - \beta \int_\Omega v^2 
 -c_2 \int_\Omega |\nabla w|^2\\
   &- \beta \int_\Omega (\nabla w \cdot \nabla z) - \beta \int_\Omega |\nabla z|^2 + (a_2 + 1) \int_\Omega w + \beta  \int_\Omega uz - b_2 \int_\Omega w^2 - \beta \int_\Omega z^2.
\end{split}
\end{equation}
Several applications of Young's inequality provide on $(0, \TM)$
\begin{equation*}\label{Young1}
2 \beta  \int_\Omega uw \leq \beta \int_\Omega u^2 + \beta \int_\Omega w^2, 
\quad \beta  \int_\Omega vw \leq \frac{\beta}{2}  \int_\Omega v^2 + \frac{\beta}{2}  \int_\Omega w^2,
\quad \beta  \int_\Omega uz \leq \frac{\beta}{2}  \int_\Omega u^2 + \frac{\beta}{2}  \int_\Omega z^2,
\end{equation*}
and for $\epsilon >0$ 
\begin{equation*}\label{Young4}
(a_1+1) \int_\Omega u \leq \epsilon \int_\Omega u^2 + \const{Ch}, \quad
(a_2+1) \int_\Omega w \leq \epsilon  \int_\Omega w^2 + \const{Chi} \quad \textrm{for all } t \in (0,\TM).
\end{equation*}
By inserting the above inequalities into estimate \eqref{SumB}, we have on $(0,\TM)$
\begin{equation*}
\begin{split}
&\frac{d}{dt} \left(\int_\Omega u + \int_\Omega w\right) + \int_\Omega u  + \int_\Omega w \leq -c_1 \int_\Omega |\nabla u|^2  - \beta \int_\Omega (\nabla u \cdot \nabla v) 
- \beta \int_\Omega |\nabla v|^2 + \left(\epsilon + \frac{3}{2}\beta - b_1\right) \int_\Omega u^2\\ 
& -c_2 \int_\Omega |\nabla w|^2 - \beta \int_\Omega (\nabla w \cdot \nabla z) - \beta \int_\Omega |\nabla z|^2 + \left(\epsilon + \frac{3}{2}\beta - b_2\right) \int_\Omega w^2+ \const{KB}\\ 
& \leq \int_\Omega \left[ -\left( \sqrt{c_1} \nabla u + \frac{\beta}{2\sqrt{c_1}} \nabla v \right) \cdot \left( \sqrt{c_1} \nabla u + \frac{\beta}{2\sqrt{c_1}} \nabla v \right) + \beta \left( \frac{\beta}{4c_1} - 1 \right) |\nabla v|^2 \right] - \left(b_1-\frac{3}{2}\beta-\epsilon\right) \int_\Omega u^2\\
&+ \int_\Omega \left[ -\left( \sqrt{c_2} \nabla w + \frac{\beta}{2\sqrt{c_2}} \nabla z \right) \cdot \left( \sqrt{c_2} \nabla w + \frac{\beta}{2\sqrt{c_2}} \nabla z \right) + \beta \left( \frac{\beta}{4c_2} - 1 \right) |\nabla z|^2 \right] - \left(b_2-\frac{3}{2}\beta-\epsilon\right) \int_\Omega w^2  + \const{KB}. 
\end{split}
\end{equation*}
Henceforth, by recalling $\alpha \max\{\chi,\xi\} \leq \beta<\min\left\{4c_1, 4 c_2,\frac{2 b_1}{3},\frac{2 b_2}{3}\right\}$ and by virtue of the arbitrariness of $\epsilon$, we obtain 
\begin{equation*}\label{C0}
\frac{d}{dt} \left(\int_\Omega u + \int_\Omega w \right)  + \int_\Omega u  + \int_\Omega w \leq \const{KB} \quad \textrm{for all } t \in (0,T_{\max}),
\end{equation*}
concluding the proof.
\end{proof}
\end{lemma}
\subsection{Temporal uniform bound on the mass for $\tau=1$}
Similarly to what we have explained in the elliptic scenario, let us show that the functional $\int_\Omega (u+w+Av^2+Bz^2)$, for proper positive $A,B$, is uniformly bounded  on $(0,\TM)$.
\begin{lemma}\label{BoundednessMassParab}
Let $\chi, \xi, a_1, a_2, c_1, c_2, h_1, h_2 >0$, $\alpha \in [0,1]$, $C_{\partial \Omega}$ be the trace constant in Lemma \ref{TraceLemma} and 
\begin{equation}\label{Estimate_bLemmaParab}
b_1> c_1 + \frac{(\xi \alpha)^2 h_2 C_{\partial \Omega}}{16 c_2}, \quad   
b_2 > c_2 + \frac{(\chi \alpha)^2 h_1 C_{\partial \Omega}}{16 c_1} \quad (\textrm{being } \alpha >0).  
\end{equation}
Then $u, w\in L^\infty((0,\TM);L^1(\Omega))$.
\begin{proof}
Arguing as in the proof of Lemma \ref{BoundednessMass} and invoking Lemma \ref{TraceLemma} (with $\psi=u$ and $\mathfrak{p}=2$), we obtain 
\begin{equation*}\label{estimateP}
\begin{split}
\frac{d}{dt} \int_\Omega u & \leq \frac{c_1}{C_{\partial \Omega}} \int_{\partial \Omega} u^2 
+ \frac{(\chi \alpha h_1)^2 C_{\partial \Omega}}{4 c_1}\int_{\partial \Omega} v^2 + a_1 \int_\Omega u - b_1 \int_\Omega u^2 - c_1 \int_\Omega  |\nabla u|^2\\
& \leq (c_1-b_1) \int_{\Omega} u^2 + \frac{(\chi \alpha h_1)^2 C_{\partial \Omega}}{4 c_1}\int_{\partial \Omega} v^2 + a_1 \int_\Omega u \quad \textrm{for all } t \in (0,\TM).
\end{split}
\end{equation*}
Moreover, applying the testing procedure to the second equation in \eqref{problem} and subsequently exploiting Young's inequality, we get 
\[
\begin{split}
\frac{(\chi \alpha)^2 h_1 C_{\partial \Omega}}{8 c_1} \frac{d}{dt} \int_{\Omega} v^2 
&= - \frac{(\chi \alpha h_1)^2 C_{\partial \Omega}}{4 c_1} \int_{\partial \Omega} v^2 - \frac{(\chi \alpha)^2 h_1 C_{\partial \Omega}}{4 c_1} \int_{\Omega} |\nabla v|^2 - \frac{(\chi \alpha)^2 h_1 C_{\partial \Omega}}{4 c_1} \int_{\Omega} v^2 + \frac{(\chi \alpha)^2 h_1 C_{\partial \Omega}}{4 c_1} \int_{\Omega} v w\\
&\leq - \frac{(\chi \alpha h_1)^2 C_{\partial \Omega}}{4 c_1} \int_{\partial \Omega} v^2 
- \frac{(\chi \alpha)^2 h_1 C_{\partial \Omega}}{4 c_1} \int_{\Omega} v^2 + \epsilon_1 \int_{\Omega} w^2 +  \left(\frac{(\chi \alpha)^2 h_1 C_{\partial \Omega}}{4 c_1}\right)^2
\frac{1}{4 \epsilon_1} \int_{\Omega} v^2 \quad \textrm{on }   (0, \TM),
\end{split}
\]
thereby yielding for all $t \in (0,\TM)$
\begin{equation}\label{Est_u_v}
\begin{split}
\frac{d}{dt} \int_{\Omega} \left( u + \frac{(\chi \alpha)^2 h_1 C_{\partial \Omega}}{8 c_1} v^2 \right) 
&\leq a_1 \int_{\Omega} u + (c_1 - b_1) \int_{\Omega} u^2 + \epsilon_1 \int_{\Omega} w^2 - \frac{(\chi \alpha)^2 h_1 C_{\partial \Omega}}{4 c_1} \left(1- \frac{(\chi \alpha)^2 h_1 C_{\partial \Omega}}{16 c_1 \epsilon_1}\right) \int_{\Omega} v^2.
\end{split}
\end{equation}
Similarly to the argument carried out above, we have the following estimate for $w$ and $z$ on $(0,\TM)$
\begin{equation*}\label{Est_w_z}
\begin{split}
\frac{d}{dt} \int_{\Omega} \left(w + \frac{(\xi \alpha)^2 h_2 C_{\partial \Omega}}{8 c_2} z^2 \right) 
&\leq a_2 \int_{\Omega} w + (c_2 - b_2) \int_{\Omega} w^2 + \epsilon_2 \int_{\Omega} u^2 - \frac{(\xi \alpha)^2 h_2 C_{\partial \Omega}}{4 c_2} \left(1- \frac{(\xi \alpha)^2 h_2 C_{\partial \Omega}}{16 c_2 \epsilon_2}\right) \int_{\Omega} z^2.
\end{split}
\end{equation*}
By adding relations \eqref{Est_u_v} and \eqref{Est_w_z}, we arrive at  
\begin{equation}\label{Est_w_z}
\begin{split}
&\frac{d}{dt} \int_{\Omega} \left( u + w+ \frac{(\chi \alpha)^2 h_1 C_{\partial \Omega}}{8 c_1} v^2 + \frac{(\xi \alpha)^2 h_2 C_{\partial \Omega}}{8 c_2} z^2\right) 
\leq a_1 \int_{\Omega} u + a_2 \int_{\Omega} w + (c_1 + \epsilon_2 - b_1) \int_{\Omega} u^2 + (c_2 +\epsilon_1 - b_2) \int_{\Omega} w^2\\
&- \frac{(\chi \alpha)^2 h_1 C_{\partial \Omega}}{4 c_1} \left(1- \frac{(\chi \alpha)^2 h_1 C_{\partial \Omega}}{16 c_1 \epsilon_1}\right) \int_{\Omega} v^2 
- \frac{(\xi \alpha)^2 h_2 C_{\partial \Omega}}{4 c_2} \left(1- \frac{(\xi \alpha)^2 h_2 C_{\partial \Omega}}{16 c_2 \epsilon_2}\right) \int_{\Omega} z^2 \quad \textrm{for all }   t \in (0, \TM).
\end{split}
\end{equation}
At this stage, in view of \eqref{Estimate_bLemmaParab}, one may proceed as in the proof of Lemma \ref{BoundednessMass} to select $\epsilon_1$ and $\epsilon_2$ in such a way that
\begin{equation*}
\frac{(\chi \alpha)^2 h_1 C_{\partial \Omega}}{16 c_1} < \epsilon_1 < b_2-c_2
\quad \textrm{and} \quad
\frac{(\xi \alpha)^2 h_2 C_{\partial \Omega}}{16 c_2} < \epsilon_2 < b_1-c_1.
\end{equation*}
Employing Young's inequality twice, we infer the following estimate
\[
\frac{d}{dt} \int_{\Omega} \left( u +w+ \frac{(\chi \alpha)^2 h_1 C_{\partial \Omega}}{8 c_1} v^2 + \frac{(\xi \alpha)^2 h_2 C_{\partial \Omega}}{8 c_2} z^2 \right) \leq -\const{Yy1} \int_{\Omega} \left( u + w+ \frac{(\chi \alpha)^2 h_1 C_{\partial \Omega}}{8 c_1} v^2 + \frac{(\xi \alpha)^2 h_2 C_{\partial \Omega}}{8 c_2} z^2 \right) + \const{Yy2} \quad \text{on } (0, \TM),
\]
which together with 
\[
\begin{split}
&\int_{\Omega} \left( u(x,0) + w(x,0) + \frac{(\chi \alpha)^2 h_1 C_{\partial \Omega}}{8 c_1} v^2(x,0)  + \frac{(\xi \alpha)^2 h_2 C_{\partial \Omega}}{8 c_2} z^2(x,0)  \right)dx\\
&=\int_{\Omega} \left( u_0(x) + w_0(x) + \frac{(\chi \alpha)^2 h_1 C_{\partial \Omega}}{8 c_1} v_0^2(x)  + \frac{(\xi \alpha)^2 h_2 C_{\partial \Omega}}{8 c_2} z_0^2(x)\right)dx
\end{split}
\]
leads to the statement.
\end{proof}
\end{lemma}
\subsection{Temporal uniform bound in $L^p(\Omega)$}
Let us turn now our attention to $L^p$-bounds, achieved by studying the temporal evolution of the functional $\frac{1}{p} \int_\Omega (u^p + w^p)$ for all $t \in (0,\TM)$ for the case $\tau=0$ and the functional $\frac1p \int_\Omega (u^p + w^p) + \frac{1}{p+1} \int_\Omega (\mathcal{C}_1 v^{p+1} + \mathcal{C}_2  z^{p+1})$ for all $t \in (0,\TM)$ for the case $\tau=1$ for suitable $\mathcal{C}_1, \mathcal{C}_2$, specified later. 
\subsubsection{The case $\tau=0$}
\begin{lemma}\label{Boundedness_u_w}
Let $n \in \N$ and the hypotheses of Lemma \ref{BoundednessMass} be complied. Then $u,w \in L^{\infty}((0, \TM); L^{p}(\Omega))$ for all $p>1$.
\begin{proof}
By testing the first equation in \eqref{problem} against $u^{p-1}$ and integrating over 
$\Omega$, we deduce
\begin{equation}\label{E1}
\begin{split}
\frac{1}{p} \frac{d}{dt} \int_{\Omega} u^p= & \int_{\partial \Omega} u^{p-1}(u_{\nu}-\chi u v_{\nu})-(p-1) \int_{\Omega} u^{p-2} |\nabla u|^2+(p-1) \chi \int_{\Omega} u^{p-1} \nabla u \cdot \nabla v\\
& +a_1 \int_{\Omega} u^p - b_1 \int_{\Omega} u^{p+1}-c_1 \int_{\Omega} u^{p-1}|\nabla u|^2\\
= & \alpha \chi h_1 \int_{\partial \Omega} u^p v - (p-1) \int_{\Omega} u^{p-2}|\nabla u|^2+\frac{(p-1)}{p} \chi \int_{\partial \Omega} u^p v_{\nu}-\frac{(p-1)}{p} \chi \int_{\Omega} u^p \Delta v\\
& +a_1  \int_{\Omega} u^p -b_1  \int_{\Omega} u^{p+1}-\frac{4 c_1}{(p+1)^2} \int_{\Omega}\left|\nabla u^{\frac{p+1}{2}}\right|^2 \textrm{ for all } t \in (0, \TM). 
\end{split}
\end{equation}
With regard to the boundary integrals, exploiting the Robin boundary condition in \eqref{problem}, we obtain 
\begin{equation}\label{E2}
\int_{\partial \Omega} u^p v_{\nu}=-h_1 \int_{\partial \Omega} u^p v \leq 0 \quad \textrm {on } (0, \TM),
\end{equation}
so that the only remaining boundary term to be handled is $\alpha \chi h_1 \int_{\partial \Omega} u^p v$. This term can be estimated by means of a combined application of Young's inequality and Lemma \ref{TraceLemma} (with $\psi=u^{\frac{p+1}{2}}, \psi=v^{\frac{p+1}{2}}$ and $\mathfrak{p}=2$), which leads to
\begin{equation}\label{E3}
\begin{split}
\alpha \chi h_1 \int_{\partial \Omega} u^p v &\leq \frac{2 c_1}{(p+1)^2 C_{\partial \Omega}} \int_{\partial \Omega} u^{p+1}+\const{1} \int_{\partial \Omega} v^{p+1}\\
& \leq \frac{2 c_1}{(p+1)^2} \int_{\Omega} u^{p+1}+\frac{2 c_1}{(p+1)^2} \int_{\Omega}\left|\nabla u^{\frac{p+1}{2}}\right|^{2}+\const{2}\left(\int_{\Omega} v^{p+1}+\int_{\Omega}\left|\nabla v^{\frac{p+1}{2}}\right|^{2}\right)\\
& \leq \frac{2 c_1}{(p+1)^2} \int_{\Omega} u^{p+1}+\frac{2 c_1}{(p+1)^2} \int_{\Omega}\left|\nabla u^{\frac{p+1}{2}}\right|^{2}+\const{3}\left(\int_{\Omega} v^{p+1}+\int_{\Omega}|\nabla v|^{p+1}\right)\\
& \leq \frac{2 c_1}{(p+1)^2} \int_{\Omega} u^{p+1}+\frac{2 c_{1}}{(p+1)^2} \int_{\Omega}\left|\nabla u^{\frac{p+1}{2}}\right|^{2}+\const{4} \int_\Omega w^{p+1} \quad \textrm{for all } t \in (0, \TM),
\end{split}
\end{equation}
where, in the last step, we have used relation \eqref{EstReg}; accordingly, the constant 
$\const{4}$ also depends on $C_{E}$. 
By exploiting the expression of $\Delta v$ in the second equation of problem \eqref{problem}, the Young inequality yields 
\begin{equation}\label{E4}
\begin{split}
- \chi \frac{(p-1)}{p}\int_{\Omega} u^p \Delta v &= - \chi \frac{(p-1)}{p}\int_{\Omega} u^p v + \chi \frac{(p-1)}{p}\int_{\Omega}u^p w \leq b_1 \int_{\Omega}u^{p+1}+ \const{5} \int_{\Omega} w^{p+1} \quad \textrm{on } (0,\TM).
\end{split}
\end{equation}
By inserting bounds \eqref{E2}, \eqref{E3} and \eqref{E4} into estimate \eqref{E1}, 
we get for all $t \in (0, \TM)$
\begin{equation}\label{E5}
\frac{1}{p} \frac{d}{dt} \int_{\Omega} u^p \leq -\frac{4(p-1)}{p^2} \int_{\Omega}\left|\nabla u^{\frac{p}{2}}\right|^2 +a_1  \int_{\Omega} u^p + \frac{2c_1}{(p+1)^2} \int_{\Omega} u^{p+1}
+ \const{7} \int_{\Omega} w^{p+1} -\frac{2 c_1}{(p+1)^2} \int_{\Omega}\left|\nabla u^{\frac{p+1}{2}}\right|^2. 
\end{equation}
Now testing the third equation in \eqref{problem} against $w^{p-1}$ and integrating over 
$\Omega$ yield
\begin{equation}\label{Ew1}
\begin{split}
\frac{1}{p} \frac{d}{dt} \int_{\Omega} w^p = & \alpha \xi h_2 \int_{\partial \Omega} w^p z - (p-1) \int_{\Omega} w^{p-2}|\nabla w|^2+\frac{(p-1)}{p} \xi \int_{\partial \Omega} w^p z_{\nu}-\frac{(p-1)}{p} \xi \int_{\Omega} w^p \Delta z\\
& +a_2  \int_{\Omega} w^p -b_2  \int_{\Omega} w^{p+1}-\frac{4 c_2}{(p+1)^2} \int_{\Omega}\left|\nabla w^{\frac{p+1}{2}}\right|^2 \quad \textrm{on } (0, \TM). 
\end{split}
\end{equation}
Similar considerations to those made above for the integrals over the boundary of $\Omega$ lead to 
\begin{equation}\label{Ew2}
\int_{\partial \Omega} w^p z_{\nu}=-h_2 \int_{\partial \Omega} w^p z \leq 0 \quad \textrm {for all } t \in (0, \TM),
\end{equation}
and employing again estimate \eqref{EstReg}
\begin{equation}\label{Ew3}
\begin{split}
\alpha \xi h_2 \int_{\partial \Omega} w^p z & \leq \frac{2 c_2}{(p+1)^2} \int_{\Omega} w^{p+1}+\frac{2 c_2}{(p+1)^2} \int_{\Omega}\left|\nabla w^{\frac{p+1}{2}}\right|^{2}+\const{3w}\left(\int_{\Omega} z^{p+1}+\int_{\Omega}|\nabla z|^{p+1}\right)\\
& \leq \frac{2 c_2}{(p+1)^2} \int_{\Omega} w^{p+1}+\frac{2 c_2}{(p+1)^2} \int_{\Omega}\left|\nabla w^{\frac{p+1}{2}}\right|^{2}+\const{w4} \int_\Omega u^{p+1} \quad \textrm{on } (0,\TM).
\end{split}
\end{equation}
By using the expression for $\Delta z$ from the fourth equation in problem \eqref{problem}, Young's inequality gives
\begin{equation}\label{Ew4}
\begin{split}
- \xi \frac{(p-1)}{p}\int_{\Omega} w^p \Delta z &= - \xi \frac{(p-1)}{p}\int_{\Omega} w^p z + \xi \frac{(p-1)}{p}\int_{\Omega} w^p u \leq \const{5w} \int_{\Omega}u^{p+1}+ b_2 \int_{\Omega} w^{p+1} \quad \textrm{for all } t \in (0, \TM).
\end{split}
\end{equation}
Plugging bounds \eqref{Ew2}, \eqref{Ew3} and \eqref{Ew4} into estimate \eqref{Ew1} infers
\begin{equation}\label{Ew5}
\frac{1}{p} \frac{d}{dt} \int_{\Omega} w^p \leq -\frac{4(p-1)}{p^2} \int_{\Omega}\left|\nabla w^{\frac{p}{2}}\right|^2 +a_2  \int_{\Omega} w^p + \const{6w} \int_{\Omega} u^{p+1} + \frac{2 c_2}{(p+1)^2} \int_{\Omega} w^{p+1} -\frac{2 c_2}{(p+1)^2} \int_{\Omega}\left|\nabla w^{\frac{p+1}{2}}\right|^2  \quad \textrm{on } (0, \TM). 
\end{equation}
A double application of Young's inequality gives
\begin{equation}\label{Y_uw}
\int_{\Omega} u^p \leq \int_{\Omega} u^{p+1} + \const{up} \quad \textrm{and} \quad 
\int_{\Omega} w^p \leq \int_{\Omega} w^{p+1} + \const{wp} \quad \textrm{for all } t \in (0,\TM).
\end{equation}
Now we add the estimates \eqref{E5} and \eqref{Ew5}, and we take into account bounds \eqref{Y_uw}, so entailing on $(0,\TM)$
\begin{equation}\label{E6}
\begin{split}
\frac{1}{p} \frac{d}{dt} \left(\int_{\Omega} u^p + \int_{\Omega} w^p \right) 
\leq& -\frac{4(p-1)}{p^2} \int_{\Omega}\left|\nabla u^{\frac{p}{2}}\right|^2 + \const{6wu} \int_{\Omega} u^{p+1} -\frac{2 c_1}{(p+1)^2} \int_{\Omega}\left|\nabla u^{\frac{p+1}{2}}\right|^2\\ 
&-\frac{4(p-1)}{p^2} \int_{\Omega}\left|\nabla w^{\frac{p}{2}}\right|^2 + \const{7wu} \int_{\Omega} w^{p+1} -\frac{2 c_2}{(p+1)^2} \int_{\Omega}\left|\nabla w^{\frac{p+1}{2}}\right|^2 + \const{su}. 
\end{split}
\end{equation}
Now, we manipulate the integral terms $\const{6wu} \int_{\Omega} u^{p+1}$ and $\const{7wu} \int_{\Omega} w^{p+1}$ by a combined application of Young's and the Gagliardo--Nirenberg inequalities. By exploiting the boundedness of the mass for $u$ and $w$ given in Lemma \ref{BoundednessMass} and setting $\theta_1:=\frac{p}{p+\frac{2}{n}} \in(0,1)$, we have
\begin{equation}\label{E_GN_u}
\begin{split}
\const{6wu} \int_{\Omega} u^{p+1} &= \const{6wu} \|u^{\frac{p+1}{2}}\|^2 _{L^2(\Omega)} \leq \const{GN1} \|\nabla u^{\frac{p+1}{2}}\|^{2 \theta_1}_{L^2(\Omega)}\|u^{\frac{p+1}{2}}\|^{2(1 -\theta_1)}_{L^{\frac{2}{p+1}}(\Omega)} + \const{GN1} 
\|u^{\frac{p+1}{2}}\|_{L^{\frac{2}{p+1}}(\Omega)}^2\\
& \leq \const{GN2}\left(\int_{\Omega}|\nabla u^{\frac{p+1}{2}}|^{2}\right)^{\theta_{1}} +\const{GN3} \leq \frac{2 c_1}{(p+1)^2} \int_{\Omega}\left|\nabla u^{\frac{p+1}{2}}\right|^2 + \const{GN4} \quad \textrm{for all } t \in (0,\TM)
\end{split}
\end{equation}
and 
\begin{equation}\label{E_GN_w}
\begin{split}
\const{7wu} \int_{\Omega} w^{p+1} &= \const{7wu} \|w^{\frac{p+1}{2}}\|^2 _{L^2(\Omega)} \leq \const{GN11} \|\nabla w^{\frac{p+1}{2}}\|^{2 \theta_1}_{L^2(\Omega)}\|w^{\frac{p+1}{2}}\|^{2(1 -\theta_1)}_{L^{\frac{2}{p+1}}(\Omega)} + \const{GN11} 
\|w^{\frac{p+1}{2}}\|_{L^{\frac{2}{p+1}}(\Omega)}^2\\
& \leq \const{GN22}\left(\int_{\Omega}|\nabla w^{\frac{p+1}{2}}|^{2}\right)^{\theta_{1}} +\const{GN33} \leq \frac{2 c_2}{(p+1)^2} \int_{\Omega}\left|\nabla w^{\frac{p+1}{2}}\right|^2 + \const{GN44} \quad \textrm{on } (0,\TM).
\end{split}
\end{equation}
Inserting bounds \eqref{E_GN_u} and \eqref{E_GN_w} into estimate \eqref{E6} gives 
\begin{equation}\label{E7}
\frac{1}{p} \frac{d}{dt} \left(\int_{\Omega} u^p + \int_{\Omega} w^p \right) 
\leq -\frac{4(p-1)}{p^2} \int_{\Omega}\left|\nabla u^{\frac{p}{2}}\right|^2 -\frac{4(p-1)}{p^2} \int_{\Omega}\left|\nabla w^{\frac{p}{2}}\right|^2 + \const{fin} \quad \textrm{for all } t \in (0,\TM). 
\end{equation}
By invoking the boundedness of the mass for $u$ and $w$ stated in Lemma \ref{BoundednessMass} and setting $\theta_2:=\frac{p-1}{p -1+\frac{2}{n}} \in(0,1)$, 
again a combination of Young's and the Gagliardo--Nirenberg inequalities provides
\begin{equation}\label{E_GN_up}
\begin{split}
\frac{1}{p} \int_{\Omega} u^p &= \frac{1}{p} \|u^{\frac{p}{2}}\|^2 _{L^2(\Omega)} \leq \const{GN1u} \|\nabla u^{\frac{p}{2}}\|^{2 \theta_2}_{L^2(\Omega)}\|u^{\frac{p}{2}}\|^{2(1 -\theta_2)}_{L^{\frac{2}{p}}(\Omega)} + \const{GN1u} 
\|u^{\frac{p}{2}}\|_{L^{\frac{2}{p}}(\Omega)}^2\\
& \leq \const{GN2u}\left(\int_{\Omega}|\nabla u^{\frac{p}{2}}|^{2}\right)^{\theta_2} +\const{GN3u} \leq \frac{4(p-1)}{p^2} \int_{\Omega}\left|\nabla u^{\frac{p}{2}}\right|^2 + \const{GN4u} \quad \textrm{on } (0,\TM)
\end{split}
\end{equation}
and 
\begin{equation}\label{E_GN_wp}
\begin{split}
\frac{1}{p} \int_{\Omega} w^p &= \frac{1}{p} \|w^{\frac{p}{2}}\|^2 _{L^2(\Omega)} \leq \const{GN1w} \|\nabla w^{\frac{p}{2}}\|^{2 \theta_2}_{L^2(\Omega)}\|w^{\frac{p}{2}}\|^{2(1 -\theta_2)}_{L^{\frac{2}{p}}(\Omega)} + \const{GN1w} 
\|w^{\frac{p}{2}}\|_{L^{\frac{2}{p}}(\Omega)}^2\\
& \leq \const{GN2w}\left(\int_{\Omega}|\nabla w^{\frac{p}{2}}|^{2}\right)^{\theta_2} +\const{GN3w} \leq \frac{4(p-1)}{p^2} \int_{\Omega}\left|\nabla w^{\frac{p}{2}}\right|^2 + \const{GN4w} \quad \textrm{for all } t \in (0,\TM).
\end{split}
\end{equation}
By virtue of bounds \eqref{E_GN_up} and \eqref{E_GN_wp}, we deduce from estimate \eqref{E7}
\begin{equation*}
\frac{d}{dt} \left(\int_{\Omega} u^p + \int_{\Omega} w^p \right) 
\leq \const{cc} -\int_{\Omega} u^p -\int_{\Omega} w^p \quad \textrm{on } (0,\TM), 
\end{equation*}
which implies $u,w \in L^{\infty}((0, \TM); L^{p}(\Omega))$ for all $p>1$.
\end{proof}
\end{lemma}
\subsubsection{The case $\tau=1$}
\begin{lemma}\label{BoundednessParab_u_w}
Let $n \in \N$ and the hypotheses of Lemma \ref{BoundednessMassParab} be complied. Then $u,w \in L^{\infty}((0, \TM); L^p(\Omega))$ for all $p>1$.
\begin{proof}
Proceeding analogously to the argument employed in \eqref{E1}, and additionally incorporating \eqref{E2} and the first inequality in \eqref{E3} with $\const{1}=h_1 \mathcal{C}_1$, and observing that $\left(a_1+\frac{1}{p}\right) \int_\Omega u^p \leq b_1 \int_\Omega u^{p+1} + \const{GGG1}$, we derive the following estimate on $(0,\TM)$
\begin{equation}\label{Ep1}
\begin{split}
\frac{1}{p} \frac{d}{dt} \int_{\Omega} u^p + \frac{1}{p} \int_{\Omega} u^p \leq& 
\frac{2 c_1}{(p+1)^2 C_{\partial \Omega}} \int_{\partial \Omega} u^{p+1}+\mathcal{C}_1 h_1 \int_{\partial \Omega} v^{p+1} - (p-1) \int_{\Omega} u^{p-2}|\nabla u|^2
-\frac{(p-1)}{p} \chi \int_{\Omega} u^p \Delta v\\
& + \left(a_1+\frac{1}{p}\right)  \int_{\Omega} u^p -b_1  \int_{\Omega} u^{p+1}-\frac{4 c_1}{(p+1)^2} \int_{\Omega}\left|\nabla u^{\frac{p+1}{2}}\right|^2 \\
& \leq \frac{2 c_1}{(p+1)^2 C_{\partial \Omega}} \int_{\partial \Omega} u^{p+1}+\mathcal{C}_1 h_1 \int_{\partial \Omega} v^{p+1}
-\frac{(p-1)}{p} \chi \int_{\Omega} u^p \Delta v-\frac{4 c_1}{(p+1)^2} \int_{\Omega}\left|\nabla u^{\frac{p+1}{2}}\right|^2 + \const{GGG1}, 
\end{split}
\end{equation}
where $\mathcal{C}_1:= (p+1)^{p-1} (\alpha \chi)^{p+1} \left(\frac{h_1 p}{2c_1} C_{\partial \Omega}\right)^p$.

Analogous testing procedures applied to the second equation in \eqref{problem}, together with Young's inequality, yield that
\begin{equation*}\label{Stima_v_P}
\begin{split}
\frac1{p+1} \frac{d}{dt} \int_\Omega v^{p+1} &= \int_{\partial \Omega} v^p v_\nu - p \int_\Omega v^{p-1} |\nabla v|^2- \int_\Omega v^{p+1}+ \int_\Omega w v^p \\ 
&\le - h_1 \int_{\partial \Omega} v^{p+1}- \frac1{p+1}\int_\Omega v^{p+1}+ \frac1{p+1}\int_\Omega w^{p+1}
\quad\textrm{on $(0,\TM)$},
\end{split}
\end{equation*}
which inserted in \eqref{Ep1} gives
\begin{align}\label{ene:Lp:3}
 &\frac{d}{dt} \left( \frac1p \int_\Omega u^p + \frac{\mathcal{C}_1}{p+1} \int_\Omega v^{p+1} \right) + \frac{1}{p} \int_\Omega u^p + \frac{\mathcal{C}_1}{(p+1)} \int_\Omega v^{p+1}
 \le \frac{2c_1}{(p+1)^2 C_{\partial \Omega}} \int_{\partial \Omega} u^{p+1}
 - \frac{p-1}p \chi \int_\Omega u^p \Delta v \notag \\
 &\quad\,\quad\,- \frac{4c_1}{(p+1)^2} \int_\Omega |\nabla u^{\frac{p+1}2}|^2 + \frac{\mathcal{C}_1}{(p+1)}  \int_\Omega w^{p+1} + \const{GGG1} \quad\textrm{for all $t \in (0,\TM)$}.
 \end{align}
Applying Lemma \ref{TraceLemma} with $\psi=u^{\frac{p+1}{2}}$ and $\mathfrak{p}=2$, together with Young's inequality, we obtain
\begin{equation*}\label{Parab2}
\frac{2c_1}{(p+1)^2 C_{\partial \Omega}} \int_{\partial \Omega} u^{p+1} - \frac{p-1}p \chi \int_\Omega u^p \Delta v \leq
\const{GG1} \int_{\Omega} u^{p+1} + \frac{2c_1}{(p+1)^2} \int_{\Omega} |\nabla u^{\frac{p+1}2}|^2 + \const{GVc1}
\int_\Omega |\Delta v|^{p+1} \quad \textrm{on } (0,\TM),
\end{equation*}
which plugged in \eqref{ene:Lp:3} implies 
\begin{equation}\label{Parab_u_v}
\begin{split}
&\frac{d}{dt}\left( \frac1p \int_\Omega u^p + \frac{\mathcal{C}_1}{p+1} \int_\Omega v^{p+1}\right)+ \left( \frac1p \int_\Omega u^p 
+ \frac{\mathcal{C}_1}{p+1}\int_\Omega v^{p+1} \right)\\
&\quad\, \le - \frac{2c_1}{(p+1)^2} \int_\Omega |\nabla u^{\frac{p+1}2}|^2
+ \const{GG1} \int_\Omega u^{p+1} + \const{GVc1} \int_\Omega |\Delta v|^{p+1}+ \const{Spe} \int_\Omega w^{p+1}
+ \const{GGG1} \quad \textrm{for all } t \in (0,\TM).
\end{split}
\end{equation}
In a way analogous to the previous argument, but now considering the third and fourth equations of system \eqref{problem}, we arrive at the following estimate
\begin{equation}\label{Parab_w_z}
\begin{split}
&\frac{d}{dt}\left( \frac1p \int_\Omega w^p + \frac{\mathcal{C}_2}{p+1} \int_\Omega z^{p+1}\right)+ \left( \frac1p \int_\Omega w^p 
+ \frac{\mathcal{C}_2}{p+1}\int_\Omega z^{p+1} \right)\\
&\quad\, \le - \frac{2c_2}{(p+1)^2} \int_\Omega |\nabla w^{\frac{p+1}2}|^2
+ \const{SG11} \int_\Omega w^{p+1} + \const{GVce1} \int_\Omega |\Delta z|^{p+1}+ \const{Sppe} \int_\Omega u^{p+1}
+ \const{GGGp1} \quad \textrm{on } (0,\TM),
\end{split}
\end{equation}
with $\mathcal{C}_2:= (p+1)^{p-1} (\alpha \xi)^{p+1} \left(\frac{h_2 p}{2c_2} C_{\partial \Omega}\right)^p$. 
Summing estimates \eqref{Parab_u_v} and \eqref{Parab_w_z}, we obtain for all $t \in (0,\TM)$
\begin{equation*}\label{Parab_uw_vz}
\begin{split}
&\frac{d}{dt}\left( \frac1p \int_\Omega u^p + \frac1p \int_\Omega w^p + \frac{\mathcal{C}_1}{p+1} \int_\Omega v^{p+1} + \frac{\mathcal{C}_2}{p+1} \int_\Omega z^{p+1}\right)
+ \left( \frac1p \int_\Omega u^p + \frac1p \int_\Omega w^p + \frac{\mathcal{C}_1}{p+1}\int_\Omega v^{p+1} + \frac{\mathcal{C}_2}{p+1}\int_\Omega z^{p+1} \right)\\
&\leq - \frac{2c_1}{(p+1)^2} \int_\Omega |\nabla u^{\frac{p+1}2}|^2 - \frac{2c_2}{(p+1)^2} \int_\Omega |\nabla w^{\frac{p+1}2}|^2
+ \const{GVc1} \int_\Omega |\Delta v|^{p+1}+ \const{GVce1} \int_\Omega |\Delta z|^{p+1}
+ \const{SGp1} \int_\Omega u^{p+1}+\const{Sppe} \int_\Omega w^{p+1}
+ \const{GGGSp1},
\end{split}
\end{equation*}
or equivalently 
\begin{equation*}\label{Parab1_uw_vz}
\begin{split}
&\frac{d}{dt} \left( e^t  \left(\frac1p \int_\Omega u^p + \frac1p \int_\Omega w^p + \frac{\mathcal{C}_1}{p+1} \int_\Omega v^{p+1} + \frac{\mathcal{C}_2}{p+1} \int_\Omega z^{p+1}
\right)\right) \leq - \frac{2c_1 e^t}{(p+1)^2} \int_\Omega |\nabla u^{\frac{p+1}2}|^2 - \frac{2c_2 e^t}{(p+1)^2} \int_\Omega |\nabla w^{\frac{p+1}2}|^2\\
&+ e^t \const{GVc1} \int_\Omega |\Delta v|^{p+1}+ e^t \const{GVce1} \int_\Omega |\Delta z|^{p+1}+ e^t \const{SGp1} \int_\Omega u^{p+1} 
+ e^t \const{Sppe} \int_\Omega w^{p+1}+ e^t\const{GGGSp1} \quad \textrm{on } (0,\TM).  
\end{split}
\end{equation*}
Upon integrating the previous estimate over the interval $(0,t)$, we arrive at 
\begin{equation*}\label{Parab12_uw_vz}
\begin{split}
&e^t  \left(\frac1p \int_\Omega u^p + \frac1p \int_\Omega w^p + \frac{\mathcal{C}_1}{p+1} \int_\Omega v^{p+1} + \frac{\mathcal{C}_2}{p+1} \int_\Omega z^{p+1}
\right) \\
&\leq \const{GGGpp1} - \frac{2c_1}{(p+1)^2} \int^t_0 e^s \left(\int_\Omega |\nabla u^{\frac{p+1}2}(\cdot,s)|^2\right)\,ds - \frac{2c_2}{(p+1)^2} \int^t_0 e^s \left(\int_\Omega |\nabla w^{\frac{p+1}2}(\cdot,s)|^2\right)\,ds\\ 
&+ \const{GVc1}  \int^t_0 e^s \left(\int_\Omega |\Delta v(\cdot,s)|^{p+1}\right)\,ds 
+ \const{GVce1}  \int^t_0 e^s \left(\int_\Omega |\Delta z(\cdot,s)|^{p+1}\right)\,ds\\
& + \const{SGp1}  \int^t_0 e^s \left( \int_\Omega u^{p+1}(\cdot,s) \right)\,ds
+ \const{Sppe}  \int^t_0 e^s \left(\int_\Omega w^{p+1}(\cdot,s)\right)\,ds+ \const{GGGSp1} (e^t-1) \quad \textrm{for all } t \in (0,\TM).  
\end{split}
\end{equation*}
Using \eqref{tau1} from Lemma \ref{RegParab} (with $\phi=v, z$, $q=p+1$, and $\psi=u, w$), we deduce that
\begin{equation}\label{Parab22_uw_vz}
\begin{split}
&e^t  \left(\frac1p \int_\Omega u^p + \frac1p \int_\Omega w^p + \frac{\mathcal{C}_1}{p+1} \int_\Omega v^{p+1} + \frac{\mathcal{C}_2}{p+1} \int_\Omega z^{p+1}
\right) \\
&\leq \const{GGGpp11} - \frac{2c_1}{(p+1)^2} \int^t_0 e^s \left(\int_\Omega |\nabla u^{\frac{p+1}2}(\cdot,s)|^2\right)\,ds - \frac{2c_2}{(p+1)^2} \int^t_0 e^s \left(\int_\Omega |\nabla w^{\frac{p+1}2}(\cdot,s)|^2\right)\,ds\\
& + \const{SGpp1}  \int^t_0 e^s \left( \int_\Omega u^{p+1}(\cdot,s) \right)\,ds
+ \const{Spp1e}  \int^t_0 e^s \left(\int_\Omega w^{p+1}(\cdot,s)\right)\,ds+ \const{GGGSp1}(e^t-1) \quad \textrm{on } (0,\TM).  
\end{split}
\end{equation}
Invoking relations \eqref{E_GN_u} and \eqref{E_GN_w}, we obtain, up to multiplicative constants,
\begin{equation*} \label{GNS4}
\const{SGpp1} \int_\Omega u^{p+1} \leq \frac{2c_1}{(p+1)^2} \int_\Omega |\nabla u^{\frac{p+1}{2}}|^2 + \const{As4}
\quad \textrm{for all } t \in (0,\TM),
\end{equation*}
and
\begin{equation*} \label{GNS4p}
\const{Spp1e} \int_\Omega w^{p+1} \leq \frac{2c_2}{(p+1)^2} \int_\Omega |\nabla w^{\frac{p+1}{2}}|^2 + \const{As4p}
\quad \textrm{on }(0,\TM),
\end{equation*}
which inserted into \eqref{Parab22_uw_vz}, yields 
\[
e^t  \left(\frac1p \int_\Omega u^p + \frac1p \int_\Omega w^p + \frac{\mathcal{C}_1}{p+1} \int_\Omega v^{p+1} + \frac{\mathcal{C}_2}{p+1} \int_\Omega z^{p+1}
\right) \le \const{GGGpp11}+ \const{Miso} (e^t-1) \quad \textrm{for all } t\in(0,\TM),
\]
thus reaching the claim.
\end{proof}
\end{lemma}
\begin{remark}[Details on the cases $\tau_1=0$ and $\tau_2=1$, and vice versa, for problem \eqref{problemRemark}]\label{Details}
As anticipated in Remark~\ref{ConsiderationOpenProbl}, we now provide some indications on the validity of Theorem \ref{theoremlocalParab} in the case $\tau_1=0$ and $\tau_2=1$ for problem \eqref{problemRemark}. The complementary case $\tau_1=1$ and $\tau_2=0$ will not be discussed, since the argument is entirely analogous. As for the mass control, this follows from the analysis of the time evolution of the functional
\[
\int_\Omega \left(u+w+\frac{(\xi \alpha)^2 h_2 C_{\partial \Omega}}{8 c_2} z^2\right) \quad \textrm{on} \; (0,\TM),
\]
which yields to a similar expression than \eqref{Est_w_z}.

On the other hand, this mass control can be subsequently used to establish the $L^p$-estimates. More precisely, by exploiting the functional
\[
 \frac{1}{p}\int_{\Omega} u^p
+ \frac{1}{p}\int_{\Omega} w^p
+ \frac{\mathcal{C}_2}{p+1}\int_{\Omega} z^{p+1} \quad \textrm{on} \; (0,\TM),
\]
a bound of the type \eqref{Parab22_uw_vz} can be achieved.  
\label{page}
\end{remark}
\section{Proofs of the theorems}
With the preparations made above, we are now in a position to prove the stated results.
\subsection{Proof of Theorem \ref{theoremlocal}}
By virtue of assumptions \eqref{Estimate_b} or \eqref{Estimate_Bag0}, we can apply Lemma \ref{Boundedness_u_w} from which it follows that $u, w \in L^{\infty}((0,\TM); L^p(\Omega))$ for all $p>1$, as a consequence, by exploiting bound \eqref{Gradient} 
(with $\phi=v$, $\psi=w$ and $\phi=z$, $\psi=u$) this implies $\nabla v, \nabla z \in L^{\infty}((0,\TM); L^\infty(\Omega))$. 
Henceforth, by invoking Lemma \ref{LemmaMoserType} (with $\psi=u$, $f=-\chi \nabla v$, $a=a_1$, $b=b_1$, $c=c_1$ and with $\psi=w$, $f=-\xi\nabla z$, $a=a_2$, $b=b_2$, $c=c_2$),  we achieve the claim.
\qed

\subsection{Proof of Theorem \ref{theoremlocalParab}}
By virtue of assumption \eqref{Estimate_bP}, Lemma \ref{BoundednessParab_u_w} implies that $u, w \in L^{\infty}((0,\TM); L^p(\Omega))$ for all $p>1$. Consequently, by exploiting estimate \eqref{tau1extension} (with $\phi=v$, $\psi=w$ and $\phi=z$, $\psi=u$), it follows that $\nabla v, \nabla z \in L^{\infty}((0,\TM); L^\infty(\Omega))$. 
From this point onward, the conclusion is obtained by applying Lemma \ref{LemmaMoserType}  in the same manner as in the elliptic case.
\qed

\subsubsection*{Acknowledgements}
SF and GV are members of the Gruppo Nazionale per l'Analisi Matematica, la Probabilit\`a e le loro Applicazioni (GNAMPA) of the Istituto Nazionale di Alta Matematica (INdAM), and they are partially supported by the research project {\em Partial Differential Equations and their role in understanding natural phenomena} (2023, CUP F23C25000080007), funded by  \href{https://www.fondazionedisardegna.it/}{Fondazione di Sardegna} and by GNAMPA-INdAM Project {\em Modelli di reazione-diffusione-trasporto: dall'analisi alle applicazioni} (CUP E53C25002010001). The work of KS is supported by the National Research Foundation of Korea (NRF) grant funded by the Korea government (MSIT, RS-2024-00350215).

\end{document}